%% file: main.tex
\documentclass{article}
\usepackage[utf8]{inputenc}

\usepackage{geometry}
\usepackage{amsmath}
\usepackage{comment}
\usepackage{leftidx}
\usepackage{multicol}
\usepackage{amssymb}
\usepackage{stmaryrd}
\usepackage{accents}
\usepackage{mathrsfs}
\usepackage{xfrac}
\usepackage{amsthm}
\usepackage{xcolor}
\usepackage{sectsty}
\usepackage{relsize}
\usepackage{float}

\usepackage{Mypackage}

\usepackage{indentfirst}

\usepackage{tikz}
\tikzset{shorten <>/.style={shorten >=#1,shorten <=#1}}

\usepackage{tikz-cd}
\newcounter{nodemaker}
\tikzset{Rightarrow/.style={double equal sign distance,>={Implies},->},
triple/.style={-,preaction={draw,Rightarrow}},
quadruple/.style={preaction={draw,Rightarrow,shorten >=0pt},shorten >=1pt,-,double,double
distance=0.2pt}}

\tikzset{%
    symbol/.style={%
        draw=none,
        every to/.append style={%
            edge node={node [sloped, allow upside down, auto=false]{$#1$}}}
    }
}
 \usepackage{quiver}

\sectionfont{\large}
\subsectionfont{\normalsize}

\usepackage{hyperref}
\hypersetup{
    colorlinks=true,
    linkcolor=teal,
    filecolor=magenta,      
    urlcolor=cyan,
    citecolor=teal
}

\usepackage{cleveref}

\newtheorem{theorem}{Theorem}[subsection]
\theoremstyle{proposition}
\newtheorem{proposition}[theorem]{Proposition}
\newtheorem{corollary}[theorem]{Corollary}
\newtheorem{corollary'}[theorem]{Corollary}
\newtheorem{lemma}[theorem]{Lemma}
\theoremstyle{definition}
\newtheorem{definition}[theorem]{Definition}
\newtheorem{example}[theorem]{Example}
\theoremstyle{definition}
\newtheorem{remark}[theorem]{Remark}
\theoremstyle{definition}
\newtheorem{division}[theorem]{}
\theoremstyle{conjecture}

\usepackage[backend=biber,
sorting=nty
]{biblatex}
\title{On a (terminally connected, pro-etale) factorization of geometric morphisms}
\author{Olivia Caramello \and Axel Osmond}
\date{}

\begin{document}

\maketitle

\begin{abstract}
    We extend the classical (connected, etale) factorization of locally connected geometric morphisms into a (terminally connected, pro-etale) factorization for all geometric morphisms between Grothendieck topoi. We discuss properties of both classes of morphisms, particularly the relation between pro-etale geometric morphisms and the category of global elements of their inverse image; we also discuss their stability properties as well as some fibrational aspects. 
\end{abstract}

\tableofcontents


\input{Sections/Introduction}

\input{Sections/Section_1}

\input{Sections/Section_2}
\input{Sections/Section_3}

\input{Sections/Section_4}

\input{Sections/Section_5}

\section*{Acknowledgements}

The second author is grateful to Morgan Rogers for insightful discussions that helped to fix an early version of the proof of the main theorem, as well as several exchanges about the properties of pro-etale geometric morphisms. He is also grateful to Mathieu Anel for several discussions that were seminal in the formation of the most early versions of this work.

\printbibliography

\vspace{1cm}

\textsc{Olivia Caramello} 

\vspace{0.2cm}
{\small \textsc{Dipartimento di Scienza e Alta Tecnologia, Universit\`a degli Studi dell'Insubria, via Valleggio 11, 22100 Como, Italy.}\\
	\emph{E-mail address:} \texttt{olivia.caramello@uninsubria.it}}

\vspace{0.2cm}

{\small \textsc{Istituto Grothendieck,
		Corso Statuto 24, 12084 Mondovì, Italy.}\\
	\emph{E-mail address:} \texttt{olivia.caramello@igrothendieck.org}}

\vspace{0.6cm}

\textsc{Axel Osmond} 

\vspace{0.2cm}
{\small \textsc{Istituto Grothendieck,
		Corso Statuto 24, 12084 Mondovì, Italy, and Paris, France.}\\
	\emph{E-mail address:} \texttt{axelosmond@orange.fr}}

\end{document}

%% file: Sections/Introduction.tex
\section*{Introduction}


A classical result of topos theory states that any locally connected geometric morphism can be factored in a unique way as a connected geometric morphism followed by the etale geometric morphism over its object of connected components. There are two possible generalizations of this result. On one hand, keeping connected morphisms as the left class gives a \emph{(connected, disconnected)} factorization, see \cite{topologie} and \cite{lurie2009higher} (where the left class is also called \emph{algebraic}. But there is another possible generalization.  

It was proved by the first author in \cite{caramello2020denseness} that essential geometric morphisms also admit a canonical factorization through the essential image of the terminal object -- which still mimics the behavior of the functor of connected components to some degree. In general the left part of this factorization is no longer connected; however, it still is \emph{terminally connected}. In the essential world, this condition states that the essential image preserves the terminal object -- which still can be read as a ``global" form of connectedness; but a more general formulation of this condition, which makes sense for any geometric morphism, is that the inverse image lifts uniquely global elements -- which amounts in some sense as being connected from the point of view of the terminal object. Terminally connected geometric morphisms happen to be exactly those that are left orthogonal to etale geometric morphisms, which suggests they form the left class of a factorization system for all geometric morphisms.

However, extending the (terminally connected, etale) factorization beyond essential geometric morphisms requires to relax in turn the etaleness condition of the right class: we have to consider \emph{cofiltered bilimits} of etale geometric morphisms. Indeed, in full generality a geometric morphism may lack an essential image part displaying connected components; yet, its inverse image part, as a lex functor, actually always possesses a \emph{left pro-adjoint}. In some sense, although the connected components of an arbitrary geometric morphism may not be indexed by an discrete set internal to the codomain topos, they will nevertheless form a pro-discrete internal locale.

Factorizing a geometric morphism through the etale geometric morphism at a given object of its codomain amounts to providing a global element of the inverse image of this given object; similarly factorizations through pro-etale geometric morphisms correspond to cofiltered diagrams of such global elements. In particular, there is a best such factorization through the pro-etale morphism indexed by the cofiltered category of all global elements of the inverse image part; although this category is large, an accessibility argument shows that it admits a small initial subcategory, which ensures the bilimit to be well defined. In particular, any pro-etale geometric morphism can be reindexed by the category of elements of its own inverse image, which provides a canonical presentation.

If now one factorizes a geometric morphism through the pro-etale morphism indexed by all the possible global elements of the inverse image, the residual left part is then always terminally connected, for in some sense all global elements are displayed in a faithful way in the cofiltered bilimit. This begets a \emph{(terminally connected, pro-etale)} factorization for all geometric morphisms. Moreover, while it was known that locally connected and terminally connected geometric morphisms are automatically connected, we dually have that essential pro-etale geometric morphisms are actually etale as their cofiltered diagram of presentation admits an initial object, namely the unit of the adjunction between the inverse image and the essential image. Hence this factorization restricts correctly to the known ones. A peculiar manifestation of this factorization is already known in the instance of the so-called \emph{Grothendieck-Verdier localization} (or at least, a generalization of it given in \cite{localmap}), which is the (terminally connected, pro-etale) factorization at a point, and is reminiscent of the construction of the \emph{germ} at this point, that is, the cofiltered intersections of its neighborhoods -- except that now one may consider etale neighborhoods rather than merely open ones.

Remarkably, this factorization allows for a canonical factorization of lax cells. In general 2-dimensional factorization systems do not come with a canonical way to relate the factorizations of two 1-cells related by a 2-cell; equivalently, the orthogonality condition for (pseudo)squares does not extend to either lax or oplax squares. But in this very case, it happens to be so. Additionally, terminally connected morphisms enjoy a special stability property along bicomma squares.

Those properties are reminiscent of those of one of the \emph{comprehensive factorization systems} for functors -- the \emph{(initial, discrete opfibration)} factorization, to which it is indeed related in several ways. First, in \cite{caramello2020denseness} the factorization of essential morphisms involved an analog of this comprehensive factorization for comorphisms of sites -- although it is unclear whether this can generalize here. Secondly, it can be shown that pro-etale morphisms are internal discrete opfibration, a property preserved at the level of points. Although it is not true that terminally connected morphisms induce initial functors between categories of points, we will see they still behave in some sense as ``topologically initial" morphisms, for one can show that the comma topos of a terminally connected geometric morphism over a point is a connected topos, which is a form of initialness. 


This paper is articulated as follows: we first recall generalities about the (connected, etale) factorization and introduce terminally connected morphisms. In the second section, after a digression on cofiltered bilimits of Grothendieck topoi, we introduce pro-etale geometric morphisms and discuss their canonical presentation through global elements. We then show the main theorem of this paper, the promised factorization -- and also its laxness aspects. Next we establish a more intrinsic characterization of pro-etale geometric morphisms as those that generate their domain through fibers of global elements. We dually investigate properties of terminally connected geometric morphisms, in particular their relation with cardinals, and discuss the quite delicate matter of their site presentation. We end with some stability properties, and a discussion on their possible relation with initial functors. 

%% file: Sections/Section_1.tex
\section{(Connected, etale) factorization and its generalization}

In this section, we recall established facts about the connected-etale factorization for locally connected geometric morphisms and its more recent generalization to essential geometric morphisms provided in \cite{caramello2020denseness}; we also give a precise account of the relation between etale geometric morphisms and global elements, which shall play a key role throughout this paper. 

In the following, we work in the 2-category $ \GTop$ of Grothendieck topoi, geometric morphisms and geometric transformations. We shall denote topoi as $ \mathcal{E}, \mathcal{F}$... and geometric as $ f, g$... We shall denote as $\GTop/\mathcal{E}$ the pseudoslice of $\GTop$ over $ \mathcal{E}$.  

\subsection{Orthogonality and factorization for geometric morphisms}

In this subsection, we first recall some generalities on 2-dimensional orthogonality structures and factorization systems. We then discuss some of their particularities in $\GTop$. In most sources, the different (2-dimensional) factorization systems of $\GTop$ are manipulated as ordinary factorization systems of the underlying category, which is indeed innocuous in practice: for this reason, the hurry reader who is not interested in the 2-dimensional details of those factorizations systems may safely skip those prerequisites. However, because most of the structure of $\GTop$ is actually bicategorical and universal constructions in geometric contexts are often weak rather than strict, we think preferable to state exactly the bicategorical versions of orthogonality and factorization, especially considering that our factorization system shall enjoy special additional 2-dimensional properties. The main references for strict 2-dimensional factorization systems are \cite{DUPONT200365}, and \cite{street2021variation} for the bicategorical ones. \\

Let us fix a 2-category $\mathcal{K}$. Denote as $2$ the category consisting of a single arrow. We denote as $ \ps[2,\mathcal{K}]$ the 2-category of 2-functors $ f : 2 \rightarrow \mathcal{K}$, pseudonatural transformations $ \alpha : f \rightarrow g$, and natural modifications between those. A close examination shows that objects of this 2-categories are exactly 1-cells of $ \mathcal{K}$. For $ f : A \rightarrow B$ and $g : C \rightarrow D$ two 1-cells in $\mathcal{K}$, a morphism $ f \rightarrow g$ in $ \ps[2,\mathcal{K}]$ consists exactly of a \emph{pseudocommutative square}, or \emph{pseudosquare} from $f$ to $g$ is a 2-functor the data of a triple $ (u,v,\alpha) : f \Rightarrow g$ with $ \alpha : gu \stackrel{\simeq }{\Rightarrow} vf$ an invertible 2-cell as below, 
\[ 
\begin{tikzcd}[row sep=small]
A \arrow[d, "f"'] \arrow[r, "u"] \arrow[rd, "\alpha \atop \simeq", phantom] & C \arrow[d, "g"] \\
B \arrow[r, "v"']                                                           & D         
\end{tikzcd} \]

Now, a 2-cell corresponds to a morphism of pseudosquares $ (\phi, \psi) : (u, v, \alpha) \Rightarrow (u', v', \alpha')  $, which is the data of $ \phi : u \Rightarrow u' $ and $\psi : v \Rightarrow v'$ such that $ \alpha' g^*\phi = \psi^*f \alpha  $ as depicted below
\[\begin{tikzcd}[row sep=small]
	A && C \\
	B && D
	\arrow[""{name=0, anchor=center, inner sep=0}, "f"', from=1-1, to=2-1]
	\arrow[""{name=1, anchor=center, inner sep=0}, "{u'}"{description}, from=1-1, to=1-3]
	\arrow[""{name=2, anchor=center, inner sep=0}, "g", from=1-3, to=2-3]
	\arrow["{v'}"', from=2-1, to=2-3]
	\arrow[""{name=3, anchor=center, inner sep=0}, "u", curve={height=-18pt}, from=1-1, to=1-3]
	\arrow["{\alpha' \atop \simeq}"{description}, draw=none, from=0, to=2]
	\arrow["\phi", shorten <=2pt, shorten >=2pt, Rightarrow, from=3, to=1]
\end{tikzcd} = 
\begin{tikzcd}
	A && C \\
	B && D
	\arrow[""{name=0, anchor=center, inner sep=0}, "f"', from=1-1, to=2-1]
	\arrow[""{name=1, anchor=center, inner sep=0}, "g", from=1-3, to=2-3]
	\arrow[""{name=2, anchor=center, inner sep=0}, "{v'}"', curve={height=18pt}, from=2-1, to=2-3]
	\arrow["u", from=1-1, to=1-3]
	\arrow[""{name=3, anchor=center, inner sep=0}, "v"{description}, from=2-1, to=2-3]
	\arrow["{\alpha \atop \simeq}"{description}, draw=none, from=0, to=1]
	\arrow["\psi", shorten <=2pt, shorten >=2pt, Rightarrow, from=3, to=2]
\end{tikzcd}\]

One may recognize that the category of pseudosquares from $f$ to $g$ as the pseudopullback
\[\begin{tikzcd}
	{\ps[2,\mathcal{K}](f,g)} & {\mathcal{K}[A,C]} \\
	{\mathcal{K}[B,D]} & {\mathcal{K}[A,D]}
	\arrow["{\mathcal{K}[f,D]}"', from=2-1, to=2-2]
	\arrow[""{name=0, anchor=center, inner sep=0}, "{\mathcal{K}[A,g]}", from=1-2, to=2-2]
	\arrow[""{name=1, anchor=center, inner sep=0}, from=1-1, to=2-1]
	\arrow[from=1-1, to=1-2]
	\arrow["{\alpha_{f,g} \atop \simeq}"{description}, draw=none, from=0, to=1]
\end{tikzcd}\]

On the other hand, for a pair of maps $f: A \rightarrow B$ and $ g : C \rightarrow D$ in a 2-category $\mathcal{K}$, precomposition by $f$ and postcomposition by $g$ induce altogether a strictly commutative square in $\Cat$
\[\begin{tikzcd}
	{\mathcal{K}[B,C]} & {\mathcal{K}[A,C]} \\
	{\mathcal{K}[B,D]} & {\mathcal{K}[A,D]}
	\arrow["{\mathcal{K}[f,D]}"', from=2-1, to=2-2]
	\arrow[""{name=0, anchor=center, inner sep=0}, "{\mathcal{K}[A,g]}", from=1-2, to=2-2]
	\arrow[""{name=1, anchor=center, inner sep=0}, "{\mathcal{K}[B,g]}"', from=1-1, to=2-1]
	\arrow["{\mathcal{K}[f,C]}", from=1-1, to=1-2]
	\arrow["{=}"{description}, draw=none, from=0, to=1]
\end{tikzcd}\]

\begin{definition}
 We say $f$ and $g$ are respectively \emph{left and right bi-orthogonal to each other}, and denote as $ f \perp g$, if one has an equivalence of categories:
 \[ \mathcal{K}[B,C] \simeq \ps[2,\mathcal{K}](f,g) \]
 \end{definition}

This condition exactly means that for any pseudosquare $ (u, v, \alpha)$ there is an arrow $s_\alpha : B \rightarrow C$, unique up to a unique invertible 2-cell, together with a pair of invertible 2-cells $(\lambda_\alpha, \rho_\alpha)$ as depicted below:
\[\begin{tikzcd}[sep=large]
	A & C \\
	B & D
	\arrow["u", from=1-1, to=1-2]
	\arrow["f"', from=1-1, to=2-1]
	\arrow["v"', from=2-1, to=2-2]
	\arrow["g", from=1-2, to=2-2]
	\arrow[""{name=0, anchor=center, inner sep=0}, "s_\alpha"{description}, from=2-1, to=1-2]
	\arrow["{\lambda_\alpha\atop \simeq}"{description, pos=0.3}, Rightarrow, draw=none, from=1-1, to=0]
	\arrow["{\rho_\alpha \atop \simeq}"{description, pos=0.3}, Rightarrow, draw=none, from=2-2, to=0]
\end{tikzcd}\]
while for a morphism of pseudosquares $ (\phi, \psi) : (u, v, \alpha) \Rightarrow (u', v', \alpha')  $ there exists a unique 2-cell $ \delta_{(\phi,\psi)} : s_{\alpha} \Rightarrow s_{\alpha'} $ such that 
$ \alpha' \phi = \sigma^*f \alpha $ and $ \alpha' g^*\sigma = \psi \alpha $ as depicted below:
\[\begin{tikzcd}
	A && C \\
	B
	\arrow["f"', from=1-1, to=2-1]
	\arrow[""{name=0, anchor=center, inner sep=0}, "{u'}"{description}, from=1-1, to=1-3]
	\arrow[""{name=1, anchor=center, inner sep=0}, "{s_{\alpha'}}"', from=2-1, to=1-3]
	\arrow[""{name=2, anchor=center, inner sep=0}, "u", curve={height=-18pt}, from=1-1, to=1-3]
	\arrow["{\lambda_{\alpha'} \atop \simeq}"{description}, draw=none, from=1-1, to=1]
	\arrow["\phi", shorten <=2pt, shorten >=2pt, Rightarrow, from=2, to=0]
\end{tikzcd} =
\begin{tikzcd}
	A && C \\
	B
	\arrow["f"', from=1-1, to=2-1]
	\arrow["u", from=1-1, to=1-3]
	\arrow[""{name=0, anchor=center, inner sep=0}, "s_\alpha"{description}, from=2-1, to=1-3]
	\arrow[""{name=1, anchor=center, inner sep=0}, "{s_{\alpha'}}"', curve={height=22pt}, from=2-1, to=1-3]
	\arrow["\delta_{(\psi,\psi)}"', shift left=2, shorten <=3pt, shorten >=3pt, Rightarrow, from=0, to=1]
	\arrow["{\lambda_\alpha \atop \simeq}"{description}, draw=none, from=1-1, to=0]
\end{tikzcd} 
\hskip0.5cm \begin{tikzcd}
	&& C \\
	B && D
	\arrow["g", from=1-3, to=2-3]
	\arrow["{v'}"', from=2-1, to=2-3]
	\arrow[""{name=0, anchor=center, inner sep=0}, "{s_{\alpha'}}"{description}, from=2-1, to=1-3]
	\arrow[""{name=1, anchor=center, inner sep=0}, "s_\alpha", curve={height=-22pt}, from=2-1, to=1-3]
	\arrow["{\rho_{\alpha'} \atop \simeq}"{description}, draw=none, from=0, to=2-3]
	\arrow["\delta_{(\phi, \psi)}", shorten <=3pt, shorten >=3pt, Rightarrow, from=1, to=0]
\end{tikzcd} =
\begin{tikzcd}
	&& C \\
	B && D
	\arrow["g", from=1-3, to=2-3]
	\arrow[""{name=0, anchor=center, inner sep=0}, "{v'}"', curve={height=18pt}, from=2-1, to=2-3]
	\arrow[""{name=1, anchor=center, inner sep=0}, "v"{description}, from=2-1, to=2-3]
	\arrow[""{name=2, anchor=center, inner sep=0}, "s_\alpha", from=2-1, to=1-3]
	\arrow["\psi", shorten <=2pt, shorten >=2pt, Rightarrow, from=1, to=0]
	\arrow["{\rho_\alpha \atop \simeq}"{description}, draw=none, from=2, to=2-3]
\end{tikzcd}\]

\begin{definition}
    A \emph{bi-orthogonality structure} on a 2-category $ \mathcal{K}$ consists of a pair of classes of 1-cells $ (\mathcal{L},\mathcal{R})$ with $ \mathcal{L} = \, ^\perp \mathcal{R}$ and $ \mathcal{R} = \mathcal{L}^{\perp}$.
\end{definition}

\begin{proposition}
Let $\mathcal{K}$ is a 2-category endowed with a bi-orthogonality structure $ (\mathcal{L}, \mathcal{R})$. Then the left and right classes enjoy the following properties :
 \begin{multicols}{2}
 \begin{itemize}
    \item $\mathcal{L} $ is closed under composition
    \item $\mathcal{L}$ is closed under invertible 2-cell
    \item $\mathcal{L}$ contains all equivalences
    \item $\mathcal{L}$ is right-pseudocancellative\index{right-pseudocancellative}: for any invertible 2-cell
\[\begin{tikzcd}
	{C_1} & {C_2} \\
	& {C_3}
	\arrow["{l_2}"', from=1-1, to=2-2]
	\arrow["{l_1}", from=1-1, to=1-2]
	\arrow[""{name=0, anchor=center, inner sep=0}, "f", from=1-2, to=2-2]
	\arrow["\simeq"{description}, Rightarrow, draw=none, from=0, to=1-1]
\end{tikzcd}\]
    with $l_1, \; l_2 $ in $\mathcal{L}$, then $f$ also is in $\mathcal{L}$
    \item  $\mathcal{L}$ is closed under weighted bicolimits in $ \ps[{2, \mathcal{K}}]$
    \item $\mathcal{L}$ is closed under bipushout along arbitrary maps
 \end{itemize} 
      
 \columnbreak
  \begin{itemize}
      \item $\mathcal{R} $ is closed under composition
      \item $\mathcal{R}$ is closed under invertible 2-cell
      \item $\mathcal{R}$ contains all equivalences
      \item $\mathcal{R} $ is left-pseudocancellative\index{left-pseudocancellative}: for any invertible 2-cell
\[\begin{tikzcd}
	{C_1} & {C_3} \\
	{C_2}
	\arrow["{r_1}", from=1-1, to=1-2]
	\arrow["f"', from=1-1, to=2-1]
	\arrow[""{name=0, anchor=center, inner sep=0}, "{r_2}"', from=2-1, to=1-2]
	\arrow["\simeq"{description}, Rightarrow, draw=none, from=0, to=1-1]
\end{tikzcd}\]
      with $r_1, \; r_2 $ in $\mathcal{R}$, then $f$ also is in $\mathcal{R}$
      \item $\mathcal{R} $ is closed under weighted bilimits in $ \ps[{2, \mathcal{K}}]$
      \item $ \mathcal{R}$ is closed under bipullback along arbitrary maps
  \end{itemize}       
 \end{multicols}
\end{proposition}

\begin{lemma}
    In a bi-orthogonality structure one automatically has the following, then $ \mathcal{L} \cap \mathcal{R}$ consists exactly of all equivalences. 
\end{lemma}

\begin{definition}
A \emph{bifactorization system}\index{bifactorization system} on $ \mathcal{K}$ is the data of a bi-orthogonality structure $ (\mathcal{L}, \mathcal{R})$ such that for any $ f : A \rightarrow B$ in $\mathcal{K}$, there exists a pair $(l_f, r_f)$ with $ l_f \in \mathcal{L}$ and $ r_f \in \mathcal{R}$ equipped with an invertible 2-cell 
        \[ 
\begin{tikzcd}
A \arrow[rr, "f"] \arrow[rd, "l_f"'] & {} \arrow[d, "\alpha_f \atop \simeq", phantom, near start] & B \\
                                     & C_f \arrow[ru, "r_f"']                         &  
\end{tikzcd} \] 

\end{definition}

\begin{remark}
By the orthogonality of left and right maps, any such factorization shall be universal in the sense that, for any other factorization $ \alpha : rl \simeq f$ with $ l_f \in \mathcal{L}$ and $ r_f \in \mathcal{R}$, there exist a unique equivalence $i$ and a pair of invertible 2-cells $ \lambda : il \simeq l_f $ and $ \rho : r \simeq r_f i$ such that $ \alpha = \alpha_f \rho^*l r^*\lambda$ as below 
\[ 
\begin{tikzcd}[row sep= small, column sep=large]
A \arrow[rr, "f"] \arrow[rd, "l_f" description] \arrow[rdd, "l"', bend right=20, "\lambda \atop \simeq"] & {} \arrow[d, "\alpha_f \atop \simeq", phantom]         & B \\
                                                                  & C_f \arrow[ru, "r_f" description]                      &   \\
                                                                  & C \arrow[ruu, "r"', bend right=20, "\rho \atop \simeq"] \arrow[u, "i"] &  
\end{tikzcd} \]
    \end{remark}

\begin{division}\label{glosis on fact of geommorph}
  The 2-category $\GTop$ of Grothendieck topoi and geometric morphisms admits several well known orthogonal factorization systems:

\begin{itemize}
    \item the (surjection, embedding) factorization;
    \item the (hyperconnected, localic) factorization;
    \item the (dense, closed) factorization;
    \item the (pure, entire) factorization...
\end{itemize}

One can show that all those factorization systems are actually bifactorization systems, with the orthogonality property extending without problem to pseudosquares. Now, recall that geometric morphisms can be read in two ways: in the way of the direct image, they convey geometric information, while in the way of the inverse image, they convey algebraic information. Hence, there are always two ways to describe the dual properties of left and right maps in a factorization system for geometric morphisms: \begin{itemize}
    \item  from a geometric point of view, left maps are connected in some way; right maps are truncated in some way;
    \item  from an algebraic point of view, inverse image parts of right maps are generating under some operations, while inverse image parts of left maps are closed under those operations.
\end{itemize}

\end{division}

In this work, we are going to exhibit a factorization system where both those complementary aspects are particularly manifest. It shall arise as a generalization of a conceptually important factorization system that however only holds for restricted classes of geometric morphisms, the \emph{(connected, etale)} factorization of \emph{locally connected} geometric morphisms. The remaining of this section shall be devoted to the exposition of this factorization system.

\subsection{Etale geometric morphisms and (connected, etale) factorization}

It is a fondamental result of topos theory that for any topos $ \mathcal{E}$ and any object $E$ in $\mathcal{E}$ the slice $ \mathcal{E}/E$ is a topos and the pullback functor $ E^* : \mathcal{E} \rightarrow \mathcal{E}/E $ is the inverse image part of a geometric morphism $ \pi_E : \mathcal{E}/E \rightarrow \mathcal{E}$, the \emph{etale geometric morphism at $E$}.

\begin{definition}
A geometric morphism $f : \mathcal{F} \rightarrow \mathcal{E}$ is said to be \emph{etale} if there exists an object $ E$ in $\mathcal{E} $ such that one has an invertible 2-cell
\[\begin{tikzcd}
	{\mathcal{F}} && {\mathcal{E}/E} \\
	& {\mathcal{E}}
	\arrow["f"', from=1-1, to=2-2]
	\arrow[""{name=0, anchor=center, inner sep=0}, "\simeq", from=1-1, to=1-3]
	\arrow["{\pi_E}", from=1-3, to=2-2]
	\arrow["\simeq"{description}, Rightarrow, draw=none, from=0, to=2-2]
\end{tikzcd}\]
In the following, we shall denote as $ \Et$ the class of etale geometric morphisms, and $ \Et/\mathcal{E}$ the 2-category of etale geometric morphisms with fixed codomain $\mathcal{E}$ - for which we shall often say, \emph{etale over $\mathcal{E}$}. 
\end{definition}

\begin{remark}
    In practice, we shall innocently consider any etale morphism as exactly being of the form $ \mathcal{E}/E \rightarrow \mathcal{E}$, ignoring the eventual equivalence through which it is presented. 
\end{remark}

Any object $E$ of a topos $\mathcal{E}$ induces an etale geometric morphism $ \mathcal{E}/E \rightarrow \mathcal{E}$. Any $ u : E \rightarrow E'$ in a topos $\mathcal{E}$ induces a canonical invertible 2-cell in $\GTop$:
\[\begin{tikzcd}
	{\mathcal{E}/E} \\
	& {\mathcal{E}} \\
	{\mathcal{E}/E'}
	\arrow[""{name=0, anchor=center, inner sep=0}, "{\mathcal{E}/u}"', from=1-1, to=3-1]
	\arrow["{\pi_{E}}", from=1-1, to=2-2]
	\arrow["{\pi_{E'}}"', from=3-1, to=2-2]
	\arrow["\alpha_u \atop \simeq"{description}, draw=none, from=0, to=2-2]
\end{tikzcd}\]
where the invertible 2-cell $ \alpha_u$ has as inverse image part $\alpha_u^\flat : u^* E^{'*} \simeq E^*$ the canonical isomorphism provided by composition of pullbacks at any object $D$ of $\mathcal{E}$
\[  E \times^u_{E'} (E' \times D) \simeq E \times D \]

\begin{remark}
    The fact that limits are defined only up to a canonical isomorphism makes this a natural isomorphism rather than a pointwise equality. As the 2-cells of $\GTop$ should be taken as being determined by their inverse image component -- this shall be our convention -- the triangle above really is an invertible 2-cell in $\GTop$, not a strict 2-cell.  

\end{remark}




\begin{definition}
For an object $E$ in $\mathcal{E}$, a \emph{global element} of $E$ is a morphism $a : 1_\mathcal{E} \rightarrow E$ in $\mathcal{E}$, where $ 1_\mathcal{E}$ is the terminal object of $\mathcal{E}$. A \emph{morphism of global elements} between two global elements $ a : 1_\mathcal{E} \rightarrow E$ and $a' : 1_\mathcal{E} \rightarrow E'$ is a morphism $ u : E \rightarrow E'$ inducing a commutative triangle 
\[\begin{tikzcd}
	& {1_\mathcal{E}} \\
	E && {E'}
	\arrow["u"', from=2-1, to=2-3]
	\arrow["a"', from=1-2, to=2-1]
	\arrow["{a'}", from=1-2, to=2-3]
\end{tikzcd}\]
\end{definition}

\begin{remark}
    Any global element $ a : 1_\mathcal{E} \rightarrow E$ is a split-monomorphism with section the terminal map $ E \rightarrow 1_\mathcal{E}$.
\end{remark}

There is a correspondence between global sections of an etale geometric morphisms over $\mathcal{E}$ and global elements of the corresponding object. Denote as $ \GTop/\mathcal{E}$ the pseudoslice of $\GTop$ over $\mathcal{E}$, that is, the slice whose 1-cells $ f \rightarrow g$ are invertible 2-cells of $\GTop$ of the following form:
\[\begin{tikzcd}
	{\mathcal{G}} && {\mathcal{F}} \\
	& {\mathcal{E}}
	\arrow[""{name=0, anchor=center, inner sep=0}, "h", from=1-1, to=1-3]
	\arrow["g"', from=1-1, to=2-2]
	\arrow["f", from=1-3, to=2-2]
	\arrow["{\alpha \atop \simeq}"{description}, draw=none, from=0, to=2-2]
\end{tikzcd}\]

We are going to prove that factorization through etale morphisms are in correspondence with global elements:

\begin{proposition}
    For $E$ an object of $\mathcal{E}$, the homcategory $\GTop/\mathcal{E}[1_\mathcal{E}, \pi_E]$ is equivalent to the discrete category given by the set of global elements of $E$:
    \[  \GTop/\mathcal{E}[1_\mathcal{E}, \pi_E] \simeq \mathcal{E}[1_\mathcal{E}, E] \]
\end{proposition}

\begin{proof}
    The process from which one constructs a global element from a factorization is described in \cite{elephant}[Proposition B3.2.4(a)] and relies on the formulation of Diaconescu theorem in term of torsors: a geometric morphism $ a : 1_\mathcal{E} \rightarrow \pi_E$ has to be defined by a discrete fibration $ X \rightarrow E$ from a filtered internal category $X$ in $\mathcal{E}$, which must hence be isomorphic to the terminal object $1_\mathcal{E}$. Conversely for a global element $ a : 1_\mathcal{E} \rightarrow E$ one can take the pullback functor along $ a$ sending $ h : D \rightarrow E$ to $ a^*D$; stability of colimits in topoi ensures that this functors is left adjoint, while commutation of limits ensures that it preserves finite limits, hence defines the inverse image part of a geometric morphism. Moreover, cancellation of pullbacks gives a natural isomorphism between $a^*E^*D$ and $ D$ itself as the latter can be chosen as the pullback as depicted below: 
\[\begin{tikzcd}
	D & {E \times D} \\
	{1_\mathcal{E}} & E & D \\
	&& {1_\mathcal{E}}
	\arrow[from=1-1, to=1-2]
	\arrow["{!_D}"', from=1-1, to=2-1]	
	\arrow["{E^*D}"', pos=0.8, from=1-2, to=2-2]
    \arrow[Rightarrow, no head, crossing over, from=1-1, to=2-3]
	\arrow["{}", from=1-2, to=2-3]
	\arrow["\lrcorner"{anchor=center, pos=0.125}, draw=none, from=1-2, to=3-3]
	\arrow["a"{description}, from=2-1, to=2-2]
	\arrow[""{name=0, anchor=center, inner sep=0}, Rightarrow, crossing over, no head, from=2-1, to=3-3]
	\arrow["{!_E}"{description}, from=2-2, to=3-3]
	\arrow["{!_D}", from=2-3, to=3-3]
	\arrow["\lrcorner"{anchor=center, pos=0.125}, draw=none, from=1-1, to=0]
\end{tikzcd}\]
\end{proof}

\begin{remark}
    Beware that the process above does not beget strictly commuting 2-cells; more generally, one can still classify global elements with geometric morphisms that are only equivalent to etale geometric morphisms through a 2-cell of the form
\[\begin{tikzcd}
	{\mathcal{F}} && {\mathcal{E}/E} \\
	& {\mathcal{E}}
	\arrow[""{name=0, anchor=center, inner sep=0}, "{e \atop \simeq}", from=1-1, to=1-3]
	\arrow["f"', from=1-1, to=2-2]
	\arrow["{\pi_E}", from=1-3, to=2-2]
	\arrow["{\alpha\atop \simeq}"{description}, draw=none, from=0, to=2-2]
\end{tikzcd}\]
\end{remark}

\begin{division}\label{etale morphisms and global elements}
This correspondence generalizes into a correspondence between factorization of geometric morphisms through etale morphisms over $\mathcal{E}$ and {global elements} of their inverse image. Let indeed $ f : \mathcal{F} \rightarrow \mathcal{E}$ be a geometric morphism. Suppose one has an invertible 2-cell as below:
\[\begin{tikzcd}
	{\mathcal{F}} && {\mathcal{E}} \\
	& {\mathcal{E}/E}
	\arrow["a"', from=1-1, to=2-2]
	\arrow["{\pi_E}"', from=2-2, to=1-3]
	\arrow[""{name=0, anchor=center, inner sep=0}, "f", from=1-1, to=1-3]
	\arrow["\simeq"{description}, Rightarrow, draw=none, from=0, to=2-2]
\end{tikzcd}\]
Then we have a factorization of $a$ through the base change 
\[\begin{tikzcd}
	{\mathcal{F}} \\
	& {\mathcal{F}/f^*E} & {\mathcal{E}/E} \\
	& {\mathcal{F}} & {\mathcal{E}}
	\arrow["{\pi_E}", from=2-3, to=3-3]
	\arrow["f"', from=3-2, to=3-3]
	\arrow[from=2-2, to=2-3]
	\arrow[curve={height=12pt}, Rightarrow, no head, from=1-1, to=3-2]
	\arrow["{\pi_{f^*E}}"{description}, from=2-2, to=3-2]
	\arrow["\lrcorner"{anchor=center, pos=0.125}, draw=none, from=2-2, to=3-3]
	\arrow["a"{description}, curve={height=-12pt}, from=1-1, to=2-3]
	\arrow[from=1-1, to=2-2]
\end{tikzcd}\]
which also defines a global section of $ \pi_{f^*E}$, pointing a global element $ a : 1_\mathcal{F} \rightarrow f^*E$. Moreover, the inverse image of the geometric morphism $ a : \mathcal{F} \rightarrow \mathcal{E}/E$ can be explicitly described in terms of the corresponding global element: it sends any arrow $ h : D \rightarrow E$ to the pullback in $\mathcal{F}$
\[\begin{tikzcd}
	{a^*f^*D} & {f^*D} \\
	{1_\mathcal{F}} & {f^*E}
	\arrow["{f^*h}", from=1-2, to=2-2]
	\arrow["a"', from=2-1, to=2-2]
	\arrow["{a^*f^*h}"', from=1-1, to=2-1]
	\arrow[from=1-1, to=1-2]
	\arrow["\lrcorner"{anchor=center, pos=0.125}, draw=none, from=1-1, to=2-2]
\end{tikzcd}\] 
\end{division}

\begin{remark}\label{Equivalence over E between etale fact of f and element}
This correspondence manifests for each $ E$ as an isomorphism
\[ \GTop/\mathcal{E}[f,\pi_E] \simeq \mathcal{F}[1_\mathcal{F},f^*E] \]
and those isomorphisms jointly yield an equivalence between the following comma-categories:
\[  1_\mathcal{F} \downarrow f^* \simeq f \downarrow \textbf{Et}/\mathcal{E}  \]
\end{remark}

\begin{remark}
    An object $D$ in a 2-category $ \mathcal{K}$ is called \emph{discrete} if for any object $C$ the homcategory $ \mathcal{K}[C,D]$ is equivalent to a discrete category, that is, a set. Here, one can see that etale geometric morphisms are discrete objects in the pseudoslice $ \GTop/\mathcal{E}$. From their definition, etale geometric morphisms over $\mathcal{E}$ are indexed by objects of $E$; in fact, this correspondence is natural and defines an equivalence of categories 
\[ \textbf{Et}/\mathcal{E} \simeq \mathcal{E} \]
This ensures that one can manipulate $\Et/\mathcal{E}$ as an ordinary category. There is an intuitive reason for this. Objects in a topos are internal sets: but sets are the same as discrete spaces. This intuition is formalized by the fact that etale geometric morphisms over $\mathcal{E}$ are \emph{localic} and correspond to \emph{discrete internal locales} in $\mathcal{E}$, which are nothing but objects of $\mathcal{E}$. 
\end{remark}

Discreteness is the finest kind of truncation: following the geometric interpretation of factorization systems given at \cref{glosis on fact of geommorph}, this suggests that etale geometric morphisms are right orthogonal to a class of morphisms expressing a dual connectedness property. They are indeed right orthogonal to the following class of geometric morphisms, which are in some sense ``as far as possible to discrete":

\begin{definition}
    A geometric morphism $f : \mathcal{F} \rightarrow \mathcal{E}$ is said to be \emph{connected} if $ f^*$ is full and faithful.
\end{definition}

\begin{proposition}[{\cite{elephant}[Proposition C3.3.4]}]
Connected geometric morphisms are left biorthogonal to etale geometric morphisms.
\end{proposition}

Beware that not any geometric morphism that is left bi-orthogonal to etale geometric morphisms has to be connected; symmetrically, not any geometric morphism that is right bi-orthogonal to connected geometric morphism has to be etale: for instance, 2-cofiltered bilimits of etale geometric morphisms are so, though they are not etale. The complete bi-orthogonality structure involved here shall be described in the next section. However, there is a class of geometric morphisms for which connected and etale morphisms form altogether an bi-orthogonality structure:

\begin{definition}
A geometric morphism $ f : \mathcal{E} \rightarrow \mathcal{F}$ is said to be \emph{locally connected} if its inverse image part $ f^*$ admits an \emph{indexed left adjoint} $ f_!$ -- see \cite{elephant}[C3.3]. A Grothendieck topos $\mathcal{E}$ is \emph{locally connected} if its terminal map is locally connected. 

\end{definition}

\begin{remark}\label{Frobenius}
A geometric morphism is locally connected if and only if the inverse image part $f^*$ admits a left adjoint $f_!$ satisfying the \emph{Frobenius identity} $ f_!(F \times_{f^*E} f^*E') \simeq f_!F \times_E E'$ for any $ a : F \rightarrow f^*E$ in $\mathcal{F}$, $b : E' \rightarrow E$ in $\mathcal{E}$. 
\end{remark}

\begin{remark}
Any etale geometric morphism $ \pi_E : \mathcal{E}/E \rightarrow \mathcal{E}$ is locally connected, with the further left adjoint of the projection $ \pi_{E} : \mathcal{E}/E \rightarrow \mathcal{E}$ given by postcomposition with the terminal map $ !_E : E \rightarrow 1_\mathcal{E}$. However, not all connected geometric morphisms are locally connected. 
\end{remark}

In some sense, locally connected geometric morphisms are those that have a well behaved notion of object of connected components, represented by the essential image. In particular, the object $ f_!(1_\mathcal{F})$ can be seen as an object of connected components of $\mathcal{F}$ seen as a topos over $\mathcal{E}$. Then the correspondence between global elements and etale factorization yields an initial factorization corresponding to the unit $ 1_\mathcal{F} \rightarrow f^*f_!(1_\mathcal{F})$ of the terminal object:

\begin{proposition}[{\cite{elephant}[Proposition C3.3.5(i)]}]
Any locally connected geometric morphism factorizes uniquely as a connected, locally connected geometric morphism followed by an etale geometric morphism:
\[\begin{tikzcd}
	{\mathcal{F}} && {\mathcal{E}} \\
	& {\mathcal{E}/f_!(1_\mathcal{F})}
	\arrow["{c_f}"', from=1-1, to=2-2]
	\arrow[""{name=0, anchor=center, inner sep=0}, "f", from=1-1, to=1-3]
	\arrow["{\pi_{f!(1)}}"', from=2-2, to=1-3]
	\arrow["\simeq"{description}, Rightarrow, draw=none, from=0, to=2-2]
\end{tikzcd}\]
where $ (c_f)_!$ sends $ F $ to the image of the terminal map $ f_!(!_F) : f_!(F) \rightarrow f_!(1_\mathcal{F})$.
\end{proposition}

\subsection{(Terminally connected, etale) factorization of essential geometric morphisms}

This factorization still makes sense if $f_!$ only defines an ordinary, not necessarily indexed left adjoint to $f^*$, that is, when $f$ is simply an \emph{essential} geometric morphism, although the left part is no longer connected: it shall lie in a more general class of morphisms that are, intuitively, connected only from the ``point of view" of the terminal object:

\begin{definition}
A geometric morphism $f : \mathcal{F} \rightarrow \mathcal{E}$ is said to be \emph{terminally connected} if its inverse image $ f^*$ lifts uniquely global elements, that is, if one has a natural isomorphism 
\[  \mathcal{F}[1_\mathcal{F}, f^*(-)] \simeq \mathcal{E}[1_\mathcal{E}, -]  \]
\end{definition}

\begin{remark}
This condition states that not only any global element $ a : 1_\mathcal{E} \rightarrow f^*(E)$ comes from a unique global element $ \overline{a} : 1_\mathcal{E} \rightarrow E$ in $\mathcal{E}$, but also that these lifts can be done functorially, in the sense that, for any $ s : F_1 \rightarrow F_2$ and any equality of global elements $ a_2 = f^*(s)a_1$ as encoded in a triangle as below
\[\begin{tikzcd}
	& {1_\mathcal{E}} \\
	{f^*(F_1)} && {f^*(F_2)}
	\arrow["{a_1}"', from=1-2, to=2-1]
	\arrow["{f^*(s)}"', from=2-1, to=2-3]
	\arrow["{a_2}", from=1-2, to=2-3]
\end{tikzcd}\] 
then the lifts are related by the following commutative triangle:
\[\begin{tikzcd}
	& {1_\mathcal{F}} \\
	{F_1} && {F_2}
	\arrow["{\overline{a_1}}"', from=1-2, to=2-1]
	\arrow["s"', from=2-1, to=2-3]
	\arrow["{\overline{a_2}}", from=1-2, to=2-3]
\end{tikzcd}\]
In other words, we have an invertible 2-cell in the 2-category of categories

\[\begin{tikzcd}
	{\mathcal{E}} && {\mathcal{F}} \\
	& \Set
	\arrow[""{name=0, anchor=center, inner sep=0}, "{f^*}"', from=1-3, to=1-1]
	\arrow["{\Gamma_\mathcal{E}}"', from=1-1, to=2-2]
	\arrow["{\Gamma_\mathcal{F}}", from=1-3, to=2-2]
	\arrow["\simeq"{description}, Rightarrow, draw=none, from=0, to=2-2]
\end{tikzcd}\]
\end{remark}

\begin{remark}\label{relation with units}
    In the setting above, the unique lift is moreover related through the unit as in the commutative triangles: 
\[\begin{tikzcd}
	{1_\mathcal{E}} & E \\
	& {f_*f^*E}
	\arrow["\overline{a}", from=1-1, to=1-2]
	\arrow["{f_*a}"', from=1-1, to=2-2]
	\arrow["{\eta_{E}}", from=1-2, to=2-2]
\end{tikzcd} \hskip1cm 
\begin{tikzcd}
	{1_\mathcal{F}} & {f^*f_*f^*E} \\
	& {f^*E}
	\arrow["{f^*f_*f^*\overline{a}}", from=1-1, to=1-2]
	\arrow["a"', from=1-1, to=2-2]
	\arrow["{\varepsilon_{E}}", from=1-2, to=2-2]
\end{tikzcd}\] 
\end{remark}

The following observation is immediate by adjunction:

\begin{proposition}
An essential geometric morphism is terminally connected if and only if one has $ f_!(1_\mathcal{E}) \simeq 1_\mathcal{F}$. 
\end{proposition}


\begin{proposition}
A locally connected geometric morphism is terminally connected if and only if it is connected.
\end{proposition}

\begin{proof}
    The Frobenius identity of \cref{Frobenius} gives that for any $E$ in $\mathcal{E}$ one can retrieve the counit of the $f_! \dashv f^*$ adjunctions as the composite $ f_! f^*E \simeq f_! (1_\mathcal{F} \times f_*E) \simeq f_!1_\mathcal{F} \times E$. Hence, as soon as $f$ is terminally connected, this gives $f_!f^*E \simeq E$, which exhibits $f^*$ as fully faithful. 
\end{proof}

A connected geometric morphism is terminally connected; beware that terminal connectedness is a weaker condition that connectedness, as it requires only full faithfulness relatively to global elements and not arbitrary generalized elements. Hence, not all terminally connected morphisms are connected: in a sense, the terminal object ``believes" they are connected, although they might not be so from the point of view of other objects. 

\begin{proposition}
Terminally connected geometric morphisms are left biorthogonal to etale geometric morphisms.
\end{proposition}

\begin{proof}
Let be an invertible square as below
\[\begin{tikzcd}
	{\mathcal{G}} & {\mathcal{E}/E} \\
	{\mathcal{F}} & {\mathcal{E}}
	\arrow["t"', from=1-1, to=2-1]
	\arrow["f"', from=2-1, to=2-2]
	\arrow["a", from=1-1, to=1-2]
	\arrow["{\pi_{E}}", from=1-2, to=2-2]
	\arrow["\simeq"{description}, draw=none, from=1-1, to=2-2]
\end{tikzcd}\]
corresponding to a global element $ a : 1_\mathcal{G} \rightarrow t^*f^*E  $ in $\mathcal{G}$. If $ t$ is terminally connected, then this element comes uniquely from a global element $ \overline{a} : 1_\mathcal{F} \rightarrow f^*E$ in $\mathcal{F}$, whose name is a diagonalization of the pseudosquare above $ \mathcal{F} \rightarrow \mathcal{E}/E$. As any other possible diagonalization would name another antecedent of $a$, this diagonalization has to be unique. Now, let us examine the functoriality of the lift relative to morphisms of pseudosquares; a situation as below 
\[\begin{tikzcd}
	{\mathcal{G}\quad} && {\mathcal{E}/E} \\
	{\mathcal{F}} && {\mathcal{E}}
	\arrow[""{name=0, anchor=center, inner sep=0}, "a", curve={height=-18pt}, from=1-1, to=1-3]
	\arrow[""{name=1, anchor=center, inner sep=0}, "{a'}"{description}, from=1-1, to=1-3]
	\arrow["t"', from=1-1, to=2-1]
	\arrow["{\alpha' \atop\simeq}"{description}, draw=none, from=1-1, to=2-3]
	\arrow["{\pi_{E}}", from=1-3, to=2-3]
	\arrow["f'"', from=2-1, to=2-3]
	\arrow["\phi", shorten <=2pt, shorten >=2pt, Rightarrow, from=0, to=1]
\end{tikzcd} = 
\begin{tikzcd}
	{\mathcal{G}} && {\mathcal{E}/E} \\
	{\mathcal{F}} && {\mathcal{E}}
	\arrow["a", from=1-1, to=1-3]
	\arrow["t"', from=1-1, to=2-1]
	\arrow["{\alpha \atop\simeq}"{description}, draw=none, from=1-1, to=2-3]
	\arrow["{\pi_{E}}", from=1-3, to=2-3]
	\arrow[""{name=0, anchor=center, inner sep=0}, "f"{description}, from=2-1, to=2-3]
	\arrow[""{name=1, anchor=center, inner sep=0}, "{f'}"', curve={height=18pt}, from=2-1, to=2-3]
	\arrow["\psi", shorten <=2pt, shorten >=2pt, Rightarrow, from=0, to=1]
\end{tikzcd}\]
defines a triangle 
\[\begin{tikzcd}
	& {t^*f^*E} \\
	{1_\mathcal{G}} \\
	& {t^*f'^* E}
	\arrow["{t^*\psi_E^\flat}", from=1-2, to=3-2]
	\arrow["a", from=2-1, to=1-2]
	\arrow["{a'}"', from=2-1, to=3-2]
\end{tikzcd}\]
which, by functoriality of the lifting property of $t$, must come uniquely from a morphism between global elements
\[\begin{tikzcd}
	& {f^*E} \\
	{1_\mathcal{F}} \\
	& {f'^* E}
	\arrow["{\psi_E^\flat}", from=1-2, to=3-2]
	\arrow["{\overline{a}}", from=2-1, to=1-2]
	\arrow["{\overline{a'}}"', from=2-1, to=3-2]
\end{tikzcd}\]
which, in turn, exactly defines a morphism between the diagonalizations provided respectively by $\overline{a}$ and $\overline{a'}$. 
\end{proof}

\begin{proposition}[{\cite{caramello2020denseness}[Proposition 4.62]}]
Any essential geometric morphism factorizes uniquely as a terminally connected geometric morphism followed by an etale geometric morphism.
\end{proposition}

\begin{remark}
    In this case, terminal connectedness amounts for the essential image $f_!$ to preserve the terminal object. Then the factorization is obtained as in the locally connected case as 
\[\begin{tikzcd}
	{\mathcal{F}} && {\mathcal{E}} \\
	& {\mathcal{E}/f_!1_{\mathcal{F}}}
	\arrow["f", from=1-1, to=1-3]
	\arrow["{\eta_{1_{\mathcal{F}}}}"', from=1-1, to=2-2]
	\arrow["{\pi_{f_!1_{\mathcal{F}}}}"', from=2-2, to=1-3]
\end{tikzcd}\]
where the left part is the name of the unit $ \eta_{1_{\mathcal{F}}} : 1_\mathcal{F} \rightarrow f^*f_!1_\mathcal{F}$.
\end{remark}

The aim of this work is to generalize this latter result to arbitrary geometric morphisms. In absence of the essential image, the factorization shall not concentrate on a specific slice of the codomain topos; as pointed out in the above discussion about global elements, we shall have to consider rather a diagram consisting of possible factorizations through etale geometric morphisms. 

%% file: Sections/Section_2.tex
\section{Pro-etale geometric morphisms}



Generalizing the (terminally connected, etale) factorization to arbitrary geometric morphisms will involve taking cofiltered bilimit of etale geometric morphisms indexed by categories of global elements. This implies some prerequisites on cofiltered bilimits of Grothendieck topoi. 

\subsection{Cofiltered bilimits of topoi and pro-etale geometric morphisms}

Let us begin with a few recalls on cofilteredness and the associated notion of initialness. 

\begin{definition}
    Recall that a category $I$ is said to be \emph{cofiltered} if \begin{itemize}
    \item it is not empty
    \item for any to $i,j $ in $I$ there is a span $ u : k \rightarrow i$, $v : k \rightarrow j$
    \item for any parallel pair $ u,v : i \rightrightarrows j$ there is $ w : k \rightarrow i$ such that $ uw = vw$.
\end{itemize}
\end{definition}

\begin{remark}
    In particular, combining those two axioms, in a cofiltered category any cospan can be closed by a span into a commutative square. 
\end{remark}

\begin{example}
    It is a standard fact that for any geometric morphism $ f : \mathcal{F} \rightarrow \mathcal{E}$ and any object $F$ of $\mathcal{F}$, the category $F \downarrow f^*$ of $F$-indexed elements of $f^*$ is cofiltered. In particular the category of global element $ 1_\mathcal{F}\downarrow f^*$ is cofiltered.
\end{example}

Cofiltered categories have a special relation with \emph{initial} functors, which are those that left limits unchanged when precomposing diagrams along them:

\begin{definition}\label{initial functors}
    A functor $F : J \rightarrow I$ is \emph{initial} if for any $i$ in $I$, the comma category $ F \downarrow i$ is non empty and connected. If $ J$ is a full subcategory $ I$ and its inclusion is initial, we say that $J$ is initial in $I$.
\end{definition}

Connectedness amounts to the possibility to exhibit a zig-zag in $J$ connecting any two arrows from $F$ to an object $i$ in $I$. In the case where $ I$ is cofiltered, this condition simplifies:

\begin{lemma}
    Let $ I$ be cofiltered. Then a functor $ F : J \rightarrow I$ is initial if and only if for any $i$ in $J$, there is an arrow $ F(j) \rightarrow i$ and for any $ d: F(j) \rightarrow i$, $d' : F(j') \rightarrow i$ there is $t : j'' \rightarrow j $, $t':j' \rightarrow j$ such that $ dF(t) = d'F(t')$. 
\end{lemma}

\begin{lemma}\label{initial subcategory lemma}
  A full subcategory $ J$ of a cofiltered category $I$ is both initial and cofiltered as soon as any $i$ in $I$ admits a map $j \rightarrow i$.
\end{lemma}

\begin{proof}
    Suppose that any $i$ admits a map $j \rightarrow i$ with $j$ in $J$. We show that $ J \downarrow i$ is connected. Take a cospan $d: j  \rightarrow i$, $d': j' \rightarrow i$: as $I$ is cofiltered there exists a span $b: i' \rightarrow j$, $b':i' \rightarrow j'$ in $I$ such that $ db=d'b'$; now there is by hypothesis an arrow $t : j'' \rightarrow i'$, and by fullness of $J$, the composites $ bt$, $b't$ exhibit $ J \downarrow i$ as connected. Cofilteredness of $J$ is immediate.
    \end{proof}

It is well known that the bicategory of Grothendieck topoi admits small cofiltered bilimits. In \cite{dubuc2011construction} is provided the following site-theoretic description: suppose one has a cofilftered diagram of Grothendieck topoi $\mathbb{E} : I \rightarrow \GTop$ (denoting $\mathbb{E}(i)$ as $ \mathcal{E}_i$ and $ \mathbb{E}(u)$ as $ f_u : \mathcal{E}_i \rightarrow \mathcal{E}_j$ for $ u : i \rightarrow j$) with $I$ a small cofiltered category. Suppose moreover we have for each $i$ a lex site $ (\mathcal{C}_i, J_i)$ for $\mathcal{E}_i$ and for each $u$ a morphism of sites $ \overline{f}_u : \mathcal{C}_j \rightarrow \mathcal{C}_i $ such that $ f_u$ is induced by $ \overline{f}_u$: in other words, require that the diagram $\mathcal{E}$ is induced from a filtered diagram of lex sites $ \mathbb{C} : I^{\op} \rightarrow \textbf{LexSites}$ such that $ \mathbb{E} \simeq \Sh \circ \mathbb{C}$. It is known that the category of $\textbf{LexSites}$ admits filtered pseudocolimits: the pseudocolimit $ \pscolim_{i \in I^{\op}} \mathcal{C}_i$ in $\Lex$ can be endowed with a coarsest topology $  \bigvee_{i \in I} q_i(J_i) $ generated by the image of the covering families in all the $J_i$ along the pseudocolimits inclusions $q_i$; this latter makes those inclusions morphisms of lex sites, and exhibits the pseudocolimit as a pseudocolimit in $\textbf{LexSites}$. Then from \cite{dubuc2011construction} we know the following:

\begin{proposition}[{\cite{dubuc2011construction} Theorem 2.4}]\label{site for bilimit topos}
In the situation above, the bilimit in $\GTop$ of the cofiltered diagram $\mathbb{E}$ is the sheaf topos
\[ \underset{i \in I}{\bilim} \; \mathcal{E}_i \simeq \Sh(\underset{i \in I^{\op}}{\pscolim} \, \mathcal{C}_i,  \bigvee_{i \in I} q_i (J_{i})  )  \]
with the bilimit projections having as inverse image part the pseudocolimit inclusions $ q_i$. 
\end{proposition}

\begin{remark}\label{Dubuc formula with the canonical site}
As pointed out in \cite{dubuc2011construction}[Theorem 2.5], one can actually dispense oneself from choosing a \emph{small} site for the topoi and take directly the pseudocolimit of the topoi themselves equiped with their canonical topology together with the inverse image part of the diagram. In the case where the specified topology are subcanonical, we have the following:
\end{remark}

\begin{lemma}
In the situation above, if the toplogies $J_i$ are all subcanonical, then the induced topology $ \bigvee_{i \in I} q_i (J_{i})$ is subcanonical also. 
\end{lemma}

\begin{proof}
    A topology is subcanonical if it is contained in the canonical topology; hence so is the join of subcanonical topologies, as their join is the smallest one containing them. 
\end{proof}

We can now introduce the class of geometric morphisms that will play the right class in our factorization theorem:


\begin{definition}
    A geometric morphism will be called \emph{pro-etale} if if it is equivalent in $\GTop/\mathcal{E}$ to a cofiltered bilimit of etale morphisms over $\mathcal{E}$  
\[\begin{tikzcd}
	{\underset{i \in I}{\bilim} \, \mathcal{E}/E_i} && {\mathcal{E}}
	\arrow["{\underset{i \in I}{\bilim} \, \pi_{E_i}}", from=1-1, to=1-3]
\end{tikzcd}\]
for a cofiltered diagram $\mathbb{E}: I  \rightarrow \mathcal{E} $. 
\end{definition}

 \begin{remark}
    A pro-etale can hence be presented (not necessarily in a unique way) by an invertible 2-cell in $\GTop/\mathcal{E}$
\[\begin{tikzcd}
	{\mathcal{F}} && {\underset{i \in I}{\bilim} \, \mathcal{E}/E_i} \\
	& {\mathcal{E}}
	\arrow[""{name=0, anchor=center, inner sep=0}, "{e \atop \simeq}", from=1-1, to=1-3]
	\arrow["{\underset{i \in I}{\bilim} \, \pi_{E_i}}", from=1-3, to=2-2]
	\arrow["p"', from=1-1, to=2-2]
	\arrow["{\alpha \atop \simeq}"{description}, draw=none, from=0, to=2-2]
\end{tikzcd}\]
where $ e$ is a geometric equivalence, and $ \alpha$ an invertible 2-cell.

Moreover, such a $p$ comes then equiped with a canonical pseudocone provided by the pasting of the presentation 2-cell $(e, \alpha)$ with the bilimit projections
\[\begin{tikzcd}
	{\underset{i \in I}{\bilim} \, \mathcal{E}/E_i} && {\mathcal{E}/E_i} \\
	& {\mathcal{E}}
	\arrow[""{name=0, anchor=center, inner sep=0}, "{p_i}", from=1-1, to=1-3]
	\arrow["{\underset{i \in I}{\bilim} \, \pi_{E_i}}"', from=1-1, to=2-2]
	\arrow["{\pi_{E_i}}", from=1-3, to=2-2]
	\arrow["{\gamma_i \atop \simeq}"{description}, Rightarrow, draw=none, from=0, to=2-2]
\end{tikzcd}\]
with $ p_d : \mathcal{E}/u p_i \simeq p_j$ the invertible transitions 2-cell at $ d : i \rightarrow j$. 
 \end{remark}

\begin{remark}
For any geometric morphism $ f : \mathcal{F} \rightarrow \mathcal{E}$ one has an equivalence of categories 
\[ \GTop/\mathcal{E} [f, \underset{i \in I}{\bilim} \; \pi_{E_i}] \simeq \underset{i \in I}{\bilim} \;\GTop/\mathcal{E} [f,  \pi_{E_i}] \]
This means that a morphism $ (x,\alpha) : f \Rightarrow \bilim_{i \in I} \pi_{E_i} $ in $\GTop/\mathcal{E}$ defines in an essentially unique way a family $(a_i)_{i \in I}$ of global elements $ a_i : 1_\mathcal{F} \rightarrow f^*E_i$ encoded by the composites $ (p_i, \gamma_i) (e, \alpha)$ in $\GTop/\mathcal{E}$,  satisfying the coherence condition that for any $d : i \rightarrow j$ one has a commutative triangle
\[\begin{tikzcd}
	& {f^*E_i} \\
	{1_\mathcal{F}} \\
	& {f^*E_j}
	\arrow["{a_i}", from=2-1, to=1-2]
	\arrow["{a_j}"', from=2-1, to=3-2]
	\arrow["{f^*u_d}", from=1-2, to=3-2]
\end{tikzcd}\]
encoded by the inverse image part of the invertible 2-cell obtained as the whiskering 
\[\begin{tikzcd}
	{a_i^* u^*} & {a_j^*}
	\arrow["{(p_d * f)^\flat \atop \simeq}", from=1-1, to=1-2]
\end{tikzcd}\]
Moreover, the cofilteredness condition ensures that for any $i, i'$ there is a span $d : j \rightarrow i $, $d' :j \rightarrow i'$ such that $ a_i = f^*u_d a_j $ and $ a_{i'} = f^*u_{d'} a_j$, and similar conditions about the equalization of parallel cells. In particular, in the cofiltered diagram $ 1_\mathcal{F} \downarrow f^*$ such solutions are provided by the inverse images of the product and the equalizers computed in $\mathcal{E}$.
\end{remark}

\begin{corollary}\label{proetale are discrete}
    Pro-etale geometric morphisms over $ \mathcal{E}$ are discrete objects in $\GTop/\mathcal{E}$.
\end{corollary}

\begin{proof}
    We saw at \cref{Equivalence over E between etale fact of f and element} that each homcategory $ \GTop/\mathcal{E}[f, \pi_E]$ is in turn equivalent to the set $\mathcal{F}[1_\mathcal{F}, f^*E] $. But as the bilimit is indexed by a 1-category $I$, the bilimit of those homset is still equivalent to a set.
\end{proof}

\begin{remark}
Observe that in the definition of pro-etale morphisms, we work in the pseudoslice $\GTop/\mathcal{E}$, where in particular the cofiltered bilimit is computed; but as cofiltered bilimits are connected, they are preserved by the slice projection $ \GTop/\mathcal{E} \rightarrow \GTop$, hence the domain of a pro-etale geometric morphisms is the bilimit of the associated slice topoi. Moreover, its universal property is not very far from the one in the slice:  
\[ \GTop/\mathcal{E} [\mathcal{F}, \underset{i \in I}{\bilim} \; \mathcal{E}/{E_i}] \simeq \underset{i \in I}{\bilim} \;\GTop/\mathcal{E} [\mathcal{F},  \mathcal{E}/{E_i}] \]
As the bilimit cocone is a pseudococone with invertible transition 2-cells $ p_d : \mathcal{E}/u_d p_i \simeq p_j$ at $d : i \rightarrow j$ -- which are however not required to be identities -- the composite of a morphism $f : \mathcal{F} \rightarrow \bilim_{i \in I} \mathcal{E}/E_i$ with this pseudococone begets a family $(p_i f)_{i \in I}$ and each of the composites $ \pi_{E_i} p_i f$ is equivalent to $ \bilim_{i \in I} \pi_{E_i} f$ in $\GTop/\mathcal{E}$, so it defines a coherent family of element of $ \bilim_{i \in I} \pi_{E_i} f $. Conversely for a pseudocone $(f_i : \mathcal{F}\rightarrow \mathcal{E}/E_i$ each $ f_i$ code a global element of the corresponding composite $ p_i f_i$, although at first sight those composites do not define a unique morphism from $\mathcal{F}$ to $\mathcal{E}$, although they are equivalent to each other. In fact, they actually do: they are global elements of the composite $ \bilim_{i \in I} \pi_{E_i} (f_i)_{i \in I}$ where $ (f_i)_{i \in I} : \mathcal{F} \rightarrow \bilim_{i \in I}\mathcal{E}/E_i$ is the universal morphism induced by the bilimit of the slice topoi. All of this ensures that actually the pro-etale morphism $ \bilim_{i \in I} \pi_{E_i}$ and its domain $ \bilim_{i \in I} \mathcal{E}/E_i$ contains exactly the same information. 

\end{remark}

\begin{remark}\label{cofiltered bilimit of etale geometric morphisms}
Let us describe in all details what gives the expression of \cref{site for bilimit topos} in the case of the domain of a pro-etale geometric morphism. If $ \mathcal{E}$ has $(\mathcal{C},J)$ as a small lex, subcanonical site of definition, then for each object $ C$ of $\mathcal{C}$, a site for $\mathcal{E}/\hirayo_C$ can be constructed by $ (\mathcal{C}/C, J_C)$, where $ J_C$ is the relativised topology generated by pullbacks of $J$-covers along $C$ -- and this site is subcanonical since $ J$ was chosen so. Suppose now that the objects $E_i$ are chosen as arising from representable $ \hirayo_{C_i}$; for etale geometric morphisms enjoy left cancellation, all the transition morphisms in the diagram must be etale, but then they correspond to morphisms in $\mathcal{E}$ between the corresponding objects, which, by full faithfulness of the Yoneda embedding, come uniquely from morphisms $ C_i \rightarrow C_j$ in $\mathcal{C}$: we thus obtain a cofiltered diagram $(C_i)_{i \in I}$ in $\mathcal{C}$ with $u_d$ for $ d : i \rightarrow j$ in $I$ as transition morphisms.

This defines an indexed site $ I^{\op} \rightarrow \Cat$ sending $ i$ on $ (\mathcal{C}/C_i, J_{C_i})$ with the pullback functors $ u_d^* : \mathcal{C}/C_j \rightarrow \mathcal{C}/C_i$ as transition morphisms for $ d : i \rightarrow j$ - which are morphisms of lex sites by cancellation of pullbacks. Moreover, $I^{\op}$ is filtered. Now recall that we can compute the filtered pseudocolimit $ \pscolim_{i \in I^{\op}} \mathcal{C}/C_i$ in $\Cat$ as a localization of the oplax colimit $ \oplaxcolim_{i \in I} \mathcal{C}/C_i $ at the cartesian morphisms associated to the fibration $\pi : \oplaxcolim_{i \in I^{\op}} \mathcal{C}/C_i \rightarrow I$. This can be described as a category of fractions: its objects are pairs $ (i, h)$ with $ i $ in $I$ and $ h : D \rightarrow C_i$ in $\mathcal{C}/C_i$, while a morphism $ (i_0,h_0) \rightarrow (i_1,h_1)$ consists of a span of the form 
\[\begin{tikzcd}
	{(i',u_d^*h_0)} & {(i_1,h_1)} \\
	{(i_0,h_0)}
	\arrow["{(d_1,g)}", from=1-1, to=1-2]
	\arrow["{(d_0, 1_{u_{d_0}^*h_0})}"', from=1-1, to=2-1]
\end{tikzcd}\]
with $ d_0 : i' \rightarrow i_0$, $ d_1 : i' \rightarrow i_1$ in $I$ and $g$ satisfying the commutation 
\[\begin{tikzcd}
	{u_{d_0}^*D_0} & {D_1} \\
	{C_{i'}} & {C_{i_1}}
	\arrow["{u_{d_1}}"', from=2-1, to=2-2]
	\arrow["{u_{d_0}^*h_0}"', from=1-1, to=2-1]
	\arrow["{h_1}", from=1-2, to=2-2]
	\arrow["g", from=1-1, to=1-2]
\end{tikzcd}\]

The pseudocolimit $\pscolim_{i \in I^{\op}} \mathcal{C}/C_i $ is small and lex because $ I$ is cofiltered, and $\Lex$ has filtered pseudocolimits which are moreover computed in $\Cat$, see \cite{DLO}[Lemma 5.3.2] for instance. Moreover, the oplax colimit is itself equiped with a \emph{vertical topology} generated from the fiberwise topologies $\bigvee_{i \in I} \iota_i (J_{C_i})  $. One can show that the localization $ \oplaxcolim_{i \in I^{\op}} \mathcal{C}/C_i \rightarrow \pscolim_{i \in I^{\op}} \mathcal{C}/C_i$ is a lex-morphism of site which induces a topology on the pseudocolimit, which we will denote the same way. Moreover, the pseudocolimit inclusions $q_i:  \mathcal{C}/C_i \rightarrow \pscolim_{i \in I^{\op}} \mathcal{C}/C_i$ define morphisms of lex sites. 

Then the cofiltered bilimit $ \bilim_{i \in I} \mathcal{E}/E_i$ has the following presentation
\[ \underset{i \in I}{\bilim} \, \mathcal{E}/E_i \simeq \Sh(\underset{i \in I^{\op}}{\pscolim} \, \mathcal{C}/C_i, \bigvee_{i \in I} \iota_i (J_{C_i})  )  \]
where the bilimit projections are the geometric morphisms induced from the morphisms of lex site $ q_i$ at the underlying pseudocolimit inclusion. In particular, as the induced topologies on the slice $J_{C_i}$ are subcanonical, so is the generated topology.
\end{remark}

\begin{remark}
In the previous remark, we suppose that the objects $E_i$ and the morphisms between them in the diagram in $\mathcal{E}$ arise from objects and morphisms inside a small standard site for $\mathcal{E}$. For $I$ is small, it is always possible to find such a site of presentation once the diagram in $\mathcal{E}$ is given. Hence, any pro-etale geometric morphism can be presented in such a way.

However, this precaution is not necessary: we can also simply equip $\mathcal{E}$ with its canonical topology, and the induced topology $J_{E_i}$ induced on each slice by the pullback functor $ E_i^* : \mathcal{E} \rightarrow \mathcal{E}/E_i$ coincides with the canonical topology on $\mathcal{E}/E_i$ for colimits are stable in Grothendieck topoi. Now from \cref{Dubuc formula with the canonical site}, we can take as a site for the bilimit the \emph{large} site 
$\pscolim_{i \in I^{\op}} \mathcal{E}/E_i$ equiped with the smallest topology making the pseudocolimit inclusions continuous. 
\end{remark}

\begin{remark}
    One can check that the class of cartesian morphisms in $ \pscolim_{I^{\op}} \mathcal{C}/C_i$ possesses a \emph{right calculus of fractions}; let us examine carefully how it manifests. First, we know that the class of cartesian morphisms is closed under composition. For the right Ore condition: consider a situation as below (so that the right leg is cartesian and $ a$ is a morphism $ a : h_1 \rightarrow d_1^*h_0$ in $\mathcal{C}/C_{i_1}$)
\[\begin{tikzcd}
	& {(i_2, d_2^*h_0)} \\
	{(i_1, h_1)} & {(i_0, h_0)}
	\arrow["{(d_2, 1_{d_2^*h_0})}", from=1-2, to=2-2]
	\arrow["{(d_1, a)}"', from=2-1, to=2-2]
\end{tikzcd}\]
then for $I$ is cofiltered, one can close this span into a square 
\[\begin{tikzcd}
	j & {i_2} \\
	{i_1} & {i_0}
	\arrow["{d_1}"', from=2-1, to=2-2]
	\arrow["{d_2}", from=1-2, to=2-2]
	\arrow["{s_1}"', from=1-1, to=2-1]
	\arrow["{s_2}", from=1-1, to=1-2]
\end{tikzcd}\]
and then close the span above into a square 
\[\begin{tikzcd}
	{(j, s_1^*h_1)} & {(i_2, d_2^*h_0)} \\
	{(i_1, h_1)} & {(i_0, h_0)}
	\arrow["{(d_2, 1_{d_2^*h_0})}", from=1-2, to=2-2]
	\arrow["{(d_1, a)}"', from=2-1, to=2-2]
	\arrow["{(s_1, 1_{s_1^*h_1})}"', from=1-1, to=2-1]
	\arrow["{(s_2, s_1^*a)}", from=1-1, to=1-2]
\end{tikzcd}\]
where the upper leg is obtained from $ s_1^*a : s_1^*h_1 \rightarrow s_1^*d_1^*h_0 = s_2^*d_2^*h_0$ and the left leg is cartesian. 

The equalization condition is trivially fullfiled: indeed, in the case of a coequalizing fork 
\[\begin{tikzcd}
	{(i, h_0)} & {(j, d^*h_1)} & {(j', h_1)}
	\arrow["{(d_1, a_1)}", shift left=1, from=1-1, to=1-2]
	\arrow["{(d_2, a_2)}"', shift right=1, from=1-1, to=1-2]
	\arrow["{(d,1_{d^*h_1})}", from=1-2, to=1-3]
\end{tikzcd}\]
then $ dd_1=dd_2$ and $ a_1 = a_2 : h_0 \rightarrow d_1^*d^*h_1 = d_2^*d^*h_1$ so the identity of $ (i, h_0)$, which is cartesian, equalizes trivially $ (d_1,a_1) $ and $ (d_2, a_2)$. As a consequence, we can really see the pseudocolimit $ \pscolim_{I^{\op}} \mathcal{C}/C_i$ as a category of right fractions. This will allow for easier manipulation of the morphisms, and in particular global elements, in the pseudocolimit and in the bilimit topos it generates.  
\end{remark}

\subsection{Canonical indexing of a pro-etale geometric morphism}

Alhough pro-etale morphisms may be presented by a given cofiltered diagram of etale morphisms, there is always a canonical way to reindex them thanks to their category of global elements. In this subsection we explain the procedure.

For a pro-etale geometric morphism $ p : \bilim_{i \in I} \mathcal{E}/C_i$ associated to a diagram $ \mathbb{C}$ in a standard site $ (\mathcal{C},J)$ for $ \mathcal{E}$, the inverse image $p^*$ restricts to a morphism of site $ \overline{p} : \mathcal{C} \rightarrow \pscolim_{i \in I^{\op}} \mathcal{C}/C_i$ sending $ C$ to the equivalence class of the projection at any index $ [ (i, C_i^*C) ]_\sim$ and a morphism $ u: C_1 \rightarrow C_2$ to the morphism represented by any of the vertical morphisms of the form $(1_i, 1_{C_i} \times u) : (i, C_i^*C_1) \rightarrow (i, C_i^*C_2) $. Now let us explain how we can reindex the limit decomposition of a pro-etale geometric morphism on its own category of elements. 

This requires some sanity-check on the nature of global elements in $\pscolim_{i \in I^{\op}} \mathcal{C}/C_i $. There is a terminal object $ 1\pscolim_{i \in I^{\op}} \mathcal{C}/C_i$ given by the class of any terminal element in the fibers $ [ (i, 1_{C_i})]_{\sim}$ for any choice of $i$. Then the global elements are also easy to represent:

\begin{lemma}\label{presentation of global element in pseudocolim}
Any global element $1_{\pscolim_{i \in I^{\op}} \mathcal{C}/C_i} \rightarrow \overline{p}(C)$ can be represented by a vertical morphism $ (i, 1_{C_i}) \rightarrow (i, C_i^*C) $ in $ \oplaxcolim_{i \in I^{\op}} \mathcal{C}/C_i$. 
\end{lemma}

\begin{proof}
Recall that morphisms in the pseudocolimit are represented as fractions with a cartesian morphism on the left: that is, if we start with a representant $(i, C_i^*C)$ for $ \overline{p}(C)$ and a representant $(j, 1_{C_j})$ for the terminal object, then a morphism between them is a span
\[\begin{tikzcd}
	{(k, h)} & {(i, C_i^*C)} \\
	{(j, 1_{C_j})}
	\arrow[from=1-1, to=2-1]
	\arrow[from=1-1, to=1-2]
\end{tikzcd}\]
with the left leg cartesian: but then, we must have $ h = 1_{C_k}$ for transitions morphisms are lex; on the other hand, for we have a (vertical, cartesian) factorization system in the oplax colimit, the right leg can be chosen to be vertical, that is, we can chose $ k =i$. Then we have a representing morphism $(i, 1_{C_i}) \rightarrow (i, C_i^*C) $, that is, a section of the product projection $ C_i \times C \rightarrow C_i$, which is also the same as a morphism $ a : C_i \rightarrow C$. 
\end{proof}


Now, let us reindex the presenting bilimit of our pro-etale geometric morphism thanks to global elements. Define the functor $ \widetilde{p}: I \rightarrow 1_{\pscolim_{i \in I} \mathcal{C}/C_i} \downarrow \overline{p}$ as sending $ i$ to the class of the element corresponding to the diagonal
\[\begin{tikzcd}
	{[(i, 1_{C_i})]_\sim} && {[(i, C_i^*C_i)]_\sim}
	\arrow["{[(1_i, \Delta_{C_i})]_\sim}", from=1-1, to=1-3]
\end{tikzcd}\]

\begin{lemma}
The functor $ \widetilde{p}: I \rightarrow 1_{\pscolim_{i \in I} \mathcal{C}/C_i} \downarrow \overline{p}$ is initial.
\end{lemma}

\begin{proof}
Take a global element of some $\overline{p}(C)$ in the pseudocolimit, induced from a $ a : C_i \rightarrow C$ as chosen in \cref{presentation of global element in pseudocolim} as the corresponding $ [(1_i, (1_{C_i},a))]_\sim$. We must prove that the category $ \widetilde{p} \downarrow  (C,[(1_i, (1_{C_i},a))]_\sim ) $ is non-empty and connected. Non-emptiness is witnessed by the existence of the triangle over $C_i$ provided by property of the diagonal from the data given by $a$:
\[\begin{tikzcd}
	& {C_i \times C_i} \\
	{C_i} && {C_i \times C} \\
	& {C_i}
	\arrow["{C_i^*C}", from=2-3, to=3-2]
	\arrow["{1_{C_i} \times a}", from=1-2, to=2-3]
	\arrow["{C_i^*C_i}"{description, pos=0.2}, from=1-2, to=3-2]
	\arrow[Rightarrow, no head, from=2-1, to=3-2]
	\arrow["{(1_{C_i}, a)}"{description, pos=0.3}, from=2-1, to=2-3, crossing over]
	\arrow["{\Delta_{C_i}}", from=2-1, to=1-2]
\end{tikzcd}\]
For $1_{\pscolim_{i \in I} \mathcal{C}/C_i} \downarrow \overline{p}$ is cofiltered by flatness of $\overline{p}$, if suffices to prove that any two parallel morphisms of elements $ \widetilde{p}(j) \rightrightarrows (C,[(1_i, a)]_{\sim})$ are equalized by some morphism $ \widetilde{p}(d) : \widetilde{p}(j') \rightarrow \widetilde{p}(j) $. A parallel pair will be presented by a square of the following form, with the upper isomorphism relating two presentations of the terminal object suited to give vertical presentation of the corresponding global elements: 
\[\begin{tikzcd}
	{[(j, 1_{C_j})]_\sim} && {[(i, 1_{C_i})]_\sim} \\
	{[(j, C_j^*C_j)]_\sim} && {[(i, C_i^*C)]_\sim}
	\arrow["{[(1_j, \Delta_{C_j})]_{\sim}}"', from=1-1, to=2-1]
	\arrow["{[(d_1, 1_{C_j} \times b_1)]_\sim}", shift left=1, from=2-1, to=2-3]
	\arrow["{[(1_i,a)]_{\sim}}", from=1-3, to=2-3]
	\arrow["\simeq", from=1-1, to=1-3]
	\arrow["{[(d_2, 1_{C_j} \times b_2)]_\sim}"', shift right=1, from=2-1, to=2-3]
\end{tikzcd}\]
The lower parallel pair corresponds through the right calculus of fractions to a pair of spans as below:
\[\begin{tikzcd}
	& {(k_1, C_{k_1}^*C_{j})} \\
	{(j, C_j^*C_j)} && {(i, C_i^*C)} \\
	& {(k_2, C_{k_2}^*C_{j})}
	\arrow["{(d_1, 1_{C_{k_1}} \times b_1)}", shift left=1, from=1-2, to=2-3]
	\arrow["{(d_2, 1_{C_{k_2}} \times b_2)}"', shift right=1, from=3-2, to=2-3]
	\arrow["{(s_1, 1_{C_{k_1}^*C_{j}})}"', shift right=1, from=1-2, to=2-1]
	\arrow["{(s_2, 1_{C_{k_2}^*C_{j}})}", shift left=1, from=3-2, to=2-1]
\end{tikzcd}\]
with $ b_1, b_2 : C_j \rightarrow C$ a parallel pair. For $I$ is cofiltered, we know there exists a span $ t_1 :j' \rightarrow k_1$, $ t_2 : j' \rightarrow k_2$ such that $ s_1t_1 = s_2t_2$ and $ d_1t_1 = d_2t_2$. Then in $\mathcal{C}/C_{j'}$ both $ s_1$ and $ s_2$ induce the same map $ (1_{C_{j'}}, f_{s_1}f_{t_1}) = (1_{C_{j'}}, f_{s_2}f_{t_2}) : 1_{C_{j'}} \rightarrow C_{j'}^*C_j$ while the data of $(d_1, 1_{C_{k_1}} \times b_1), (d_2,1_{C_{k_2}} \times b_2)$ are transferred to a pair of morphisms in $\mathcal{C}/C_{j'}$
\[\begin{tikzcd}[sep=large]
	{C_{j'}^*C_j} & {C_{j'}^*C }
	\arrow["{f_{t_1}^*(1_{C_{k_1}} \times b_1)}", shift left=1, from=1-1, to=1-2]
	\arrow["{f_{t_2}^*(1_{C_{k_2}} \times b_2)}"', shift right=1, from=1-1, to=1-2]
\end{tikzcd}\]
But those later are actually nothing but $ 1_{C_{j'}} \times b_1$ and $1_{C_{j'}} \times b_2$ respectively, and the condition above states that one has an equality  
\[\begin{tikzcd}
	& {1_{C_{j'}}} \\
	{C_{j'}^*C_j} && {C_{j'}^*C }
	\arrow["{1_{C_{j'}} \times b_1}", shift left=1, from=2-1, to=2-3]
	\arrow["{1_{C_{j'}} \times b_2}"', shift right=1, from=2-1, to=2-3]
	\arrow["{(1_{C_{j'}}, f_{s_1}f_{t_1}) = (1_{C_{j'}}, f_{s_2}f_{t_2})}"', from=1-2, to=2-1]
	\arrow["{f_{t_1}^*f_{d_1}^*a = f_{t_2}^*f_{d_2}^*a}", from=1-2, to=2-3]
\end{tikzcd}\]
and as all those maps live in $\mathcal{C}/C_j$, $ a$ has to be uniquely induced from some $ b : C_j \rightarrow C$ and one has $ b_1 = b_2 =b$. This forces the corresponding parallel morphisms we started with to be actually equal. \end{proof}

As a consequence, though pro-etale may be presented as indexed by arbitrary small cofiltered diagrams, they always possess a canonical presentation indexed by the category of global elements of their inverse image:

\begin{corollary}\label{canonical presentation of proetale}
Any pro-etale geometric morphism $p : \mathcal{F} \rightarrow \mathcal{E}$ is equivalent in $\GTop/\mathcal{E}$ to the pro-etale geometric morphism indexed by the category of global elements of its inverse image: we have a geometric equivalence over $\mathcal{E}$
\[\begin{tikzcd}
	{\mathcal{F}} && {\underset{(E,a) \in  1_\mathcal{F} \downarrow p^*}{\bilim} \; \mathcal{E}/E} \\
	& {\mathcal{E}}
	\arrow["{(\ulcorner a \urcorner)_{(E,a) \in  1_\mathcal{F} \downarrow p^*} \atop \simeq}", from=1-1, to=1-3]
	\arrow[""{name=0, anchor=center, inner sep=0}, "p"', from=1-1, to=2-2]
	\arrow[""{name=1, anchor=center, inner sep=0}, "{\underset{ 1_\mathcal{F} \downarrow p^*}{\bilim} \; \pi_E}", from=1-3, to=2-2]
	\arrow["\simeq"', draw=none, from=1, to=0]
\end{tikzcd}\]

\end{corollary}

%% file: Sections/Section_3.tex
\section{(Terminally connected, pro-etale) factorization of geometric morphisms}

In this section, we prove that pro-etale geometric morphisms are the correct generalization of etale geometric morphisms when extending the (terminally connected, etale) factorization to arbitrary geometric morphisms. This factorization shall also be a generalization of the construction of the \emph{germ} of a topos at a point, as being obtained as a cofiltered limit of all etale factorizations of this morphism. 

\subsection{Some preliminary observations}

The following fact is an immediate consequence of closure properties of right classes under limits in the pseudoslices, since terminally connected geometric morphisms are left bi-orthogonal to etale ones:

\begin{lemma}
Pro-etale geometric morphisms are right bi-orthogonal to terminally connected geometric morphisms. 
\end{lemma}

\begin{proof}
We saw that terminally connected geometric morphisms are left bi-orthogonal to etale geometric morphisms. Since a right class in a 2-orthogonality structure is closed under bilimits, pro-etale geometric morphisms are ensured to be also right bi-orthogonal to terminally connected geometric morphisms.  
\end{proof}


We are going to consider a bilimit indexed by the category of all global elements $ 1_\mathcal{F}\downarrow f^*$ of a geometric morphism $ f : \mathcal{F} \rightarrow \mathcal{E}$; as this category is large, we must exhibit a small initial category inside of it to ensure the existence of the bilimit. We process by the following accessibility argument.


It is known, see for instance \cite{borceux1994handbook3}[Proposition 3.4.16], that Grothendieck topoi are locally presentable and can be exhibited as a lex localization for the category of $\Set$-valued sheaves over a generator of $\mu$-compact objects for some sufficiently large $\mu$. Moreover, in a locally presentable category, any object becomes compact above a certain cardinal.  

\begin{lemma}\label{cofinal cat of representable elements}
Let be $ f : \mathcal{F} \rightarrow \mathcal{E}$ a geometric morphism. Then there exists a small site $ (\mathcal{C},J)$ such that the category $ 1_\mathcal{F} \downarrow f^* \hirayo$ of global elements of the form $ 1_\mathcal{F} \rightarrow f^*(c)$ for $c$ in $\mathcal{C}$ is cofiltered and initial in the category $ 1_\mathcal{F}\downarrow f^*$.
\end{lemma}

\begin{proof}
For $ f : \mathcal{F} \rightarrow \mathcal{E}$, $ \mathcal{F}$ is in particular locally presentable, and hence, there exists some cardinal $\kappa$ such that $ 1_\mathcal{F}$ is $ \kappa$-compact. On the other side, as $\mathcal{E}$ is also locally presentable, there exists some cardinal $ \lambda $ such that $ \mathcal{E}$ is locally $\lambda$-presentable; hence, it suffices to take $\mu\geq \kappa,\lambda$, so that we simultaneously have $ 1_\mathcal{F}$ is $\mu$-compact and the subcategory $\mathcal{E}_\mu$ of $\mu$-compact objects generates $ \mathcal{E}$ under $\mu$-filtered colimits. Then, there exists a Grothendieck topology $ J_\mu$ on $ \mathcal{E}_\mu$ such that $(\mathcal{E}_\mu, J_\mu)$ is a small site of definition for $ \mathcal{E}$. In fact, $ J_\mu$ is the restriction of the canonical topology of $\mathcal{E}$ to $\mu$-compact objects. \\

But now, for any $F$ in $\mathcal{E}$, we have a $ \mu$-filtered colimit $ F \simeq \colim \, \iota_\mu \downarrow F $, where $ \iota_\mu : \mathcal{E}_\mu \hookrightarrow \mathcal{E}$ is the dense inclusion of $\mu$-compact objects. As $ 1_\mathcal{F}$ is $ \mu$-compact while $ f^* $ preserves colimits, any global element $ a : 1_\mathcal{F} \rightarrow f^*F$ lifts as follows
\[\begin{tikzcd}
	{1_\mathcal{F}} && {f^*(F)} \\
	& {f^*(K)}
	\arrow["a", from=1-1, to=1-3]
	\arrow["b"', dashed, from=1-1, to=2-2]
	\arrow["{f^*(x)}"', from=2-2, to=1-3]
\end{tikzcd}\]
for some $ x : K \rightarrow F$ with $K$ a $\mu$-compact object in $\mathcal{E}$. This exactly means that any global element $ (F,a)$ admits a morphism $ x : (K, b) \rightarrow (F,a)$ with $ K$ $\mu$-compact. But now, since $f^*$ was the inverse image part of a geometric morphism, the category of global element $ 1_\mathcal{F}\downarrow f^*$ is cofiltered. Then, by \cref{initial subcategory lemma}, $ 1_\mathcal{F} \downarrow f^* \iota_\mu$ is itself cofiltered and initial in $1_\mathcal{F} \downarrow f^*$. \end{proof}


\begin{corollary}\label{large bilimit exists}
For any geometric morphism $ f : \mathcal{F} \rightarrow \mathcal{E}$, the large cofiltered bilimit $ \bilim_{1_\mathcal{F} \downarrow f^*} \mathcal{E}/E$ exists in the 2-category of Grothendieck topoi.
\end{corollary}

\begin{proof}
    From \cref{cofinal cat of representable elements}, we know that there always exists a small cofinal category of elements of the restriction of the inverse image of $f$ to a convenient small site $ (\mathcal{C},J)$ of presentation. Then we are ensured to have an equivalence of categories 
    \[ \underset{1_\mathcal{F} \downarrow f^*}\bilim  \;\mathcal{E}/E \simeq \underset{1_\mathcal{F} \downarrow f^*\hirayo}\bilim  \; \mathcal{E}/\hirayo_C  \]
    where the later cofiltered bilimit is ensured to exist as it is small indexed. 
\end{proof}

\subsection{The factorization theorem}
 
We are now reaching the central result of this paper, establishing the promised factorization theorem:

\begin{theorem}\label{terminally connected-proetale factorization}
Any geometric morphism factorizes in an essentially unique way as a terminally connected geometric morphism followed by a pro-etale geometric morphism.
\end{theorem}

\begin{proof}
This is the main proof of this paper and deserves a quick overview of our strategy. After recalling the correspondence of etale factorizations with global elements of the inverse image, we propose a pro-etale candidate on the right, constructed as a cofiltered bilimit indexed by such global elements. Then we present it through a canonical pseudocolimit of sites, and describe global elements in this bilimit topos. The bilimit property induces a universal map on the left: we first discuss its action on objects of the pseudocolimit site, and then how it extends to arbitrary objects in the associated bilimit topos; we finally show it lifts uniquely global elements, ensuring the associated geometric morphism to be terminally connected. In a sense, the key argument is that this factorization ``displays" all possible global elements of the inverse image of the given geometric morphism, so that the residual left part must reflect them faithfully. \\

Let be $ f : \mathcal{F} \rightarrow \mathcal{E}$ be a geometric morphism. We saw at \cref{etale morphisms and global elements} that we have an equivalence of categories
\[  1_\mathcal{F} \downarrow f^* \simeq f \downarrow \textbf{Et}/\mathcal{F}  \] 

We are going to consider the following large bilimit $ \bilim_{1_\mathcal{F} \downarrow f^*} \mathcal{E}/E$ indexed over all possible global elements $a: 1_\mathcal{F} \rightarrow f^*E$ without restriction to any site of presentation for $\mathcal{E}$. From \cref{large bilimit exists}, we know this bilimit to be correctly defined. As $f^*$ is flat, $1_\mathcal{F} \downarrow f^*$ is cofiltered: hence the bilimit morphism $ \pi_f : \bilim_{1_\mathcal{F} \downarrow f^*} \mathcal{E}/E \rightarrow \mathcal{E}$ is pro-etale. Moreover, we know from \cref{Dubuc formula with the canonical site} that it is presented as 
\[  \underset{1_\mathcal{F} \downarrow f^*}\bilim  \;\mathcal{E}/E \simeq \Sh(\underset{1_\mathcal{F}\downarrow f^*}{\pscolim} \; \mathcal{E}/E, \bigvee_{1_\mathcal{F}\downarrow f^*} q_{(E,a)} (J_E)) \]
where the $ q_{(E,a)} : \mathcal{E}/E \rightarrow \pscolim_{1_\mathcal{F} \downarrow f^*} \mathcal{E}/E$ are the pseudocolimit inclusions.\\

On the other side, the correspondance above between global elements and factorizations of $f$ yields a cone $(a : f \rightarrow \pi_E)_{(E,a) \in 1_\mathcal{F} \downarrow f^*}$ in $\GTop/\mathcal{E}$, which induces uniquely a factorization 
\[\begin{tikzcd}
	{\mathcal{F}} && {\mathcal{E}} \\
	& {\underset{1_\mathcal{F} \downarrow f^*}\bilim  \;\mathcal{E}/E}
	\arrow[""{name=0, anchor=center, inner sep=0}, "f", from=1-1, to=1-3]
	\arrow["{l_f}"', from=1-1, to=2-2]
	\arrow["{\pi_f}"', from=2-2, to=1-3]
	\arrow["\simeq"{description}, draw=none, from=0, to=2-2]
\end{tikzcd}\]
together at each element $ (E,a)$ an invertible 2-cell, provided by the inversal property of the bilimit, factorizing the name of the element $a$ through the composite of $l_f$ and the pro-etale part:
\[\begin{tikzcd}
	{\mathcal{F}} && {\underset{1_\mathcal{F} \downarrow f^*}\bilim  \;\mathcal{E}/E} \\
	& {\mathcal{E}/E}
	\arrow[""{name=0, anchor=center, inner sep=0}, "{l_f}", from=1-1, to=1-3]
	\arrow["a"', from=1-1, to=2-2]
	\arrow["{p_{(E,a)}}", from=1-3, to=2-2]
	\arrow["\simeq"{description}, draw=none, from=0, to=2-2]
\end{tikzcd}\]
 


Now we must prove $ l_f$ to be terminally connected. Let us first give a word about its action at the level of the pseudocolimit site of presentation of the bilimit. Recall that the pseudocolimit $ \pscolim_{1_\mathcal{F} \downarrow f^*} \mathcal{E}/E$, as a filtered pseudocolimit of lex-categories, is computed in $\Cat$ as the localization of the oplax colimit at cartesian morphisms: in other words, its objects are equivalence classes of triples $ (E,a, h)$ with $(E,a)$ a global element of $f^*$ and $h : D \rightarrow E$ an object of $\mathcal{E}/E$. Moreover, the restriction of $l_f^*$ to this site can be computed as applying the corresponding $a^*$ to an object of the site presented by a triple $(E,a,h)$which returns a pullback
\[\begin{tikzcd}
	{l_f^*\hirayo_{[(E,a,h)]_\sim}} & {f^*D} \\
	{1_\mathcal{F}} & {f^*{E}}
	\arrow[from=1-1, to=1-2]
	\arrow["{a^*h}"', from=1-1, to=2-1]
	\arrow["\lrcorner"{anchor=center, pos=0.125}, draw=none, from=1-1, to=2-2]
	\arrow["{f^*h}", from=1-2, to=2-2]
	\arrow["a"', from=2-1, to=2-2]
\end{tikzcd}\]
which is, as explained at \cref{etale morphisms and global elements}, the pullback of $a$ and $f^*h$ over $f^*E$. Similarly, for a morphism in the oplax colimit $ (u,h) : (E_1,a_1,h_1) \rightarrow (E_2,a_2,h_2)$ corresponding to a square \[\begin{tikzcd}
	{D_1} & {D_2} \\
	{E_1} & {E_2}
	\arrow["{h_2}", from=1-2, to=2-2]
	\arrow["{h_1}"', from=1-1, to=2-1]
	\arrow["u"', from=2-1, to=2-2]
	\arrow[from=1-1, to=1-2]
\end{tikzcd}\] 
the value of $ l_f^* \hirayo_{[(u,h)]_\sim}$ is given by morphism induced from the universal property of the pullback. This expression is well defined in the sense that it does not depend on the choice of the representant. Indeed, two triples $(E_1,a_1,h_1)$ and $(E_2,a_2,h_2)$ are identified in the pseudocolimit if and only if they are related by a cartesian morphism $ (u, \id_{u^*h_2}) : (E_1,a_1,h_1) \rightarrow (E_2,a_2,h_2)$, that is, if one has a pullback
\[\begin{tikzcd}
	{D_1} & {D_2} \\
	{E_1} & {E_2}
	\arrow["{h_2}", from=1-2, to=2-2]
	\arrow["{h_1}"', from=1-1, to=2-1]
	\arrow["u"', from=2-1, to=2-2]
	\arrow[from=1-1, to=1-2]
	\arrow["\lrcorner"{anchor=center, pos=0.125}, draw=none, from=1-1, to=2-2]
\end{tikzcd}\]
where $u$ moreover provides a morphism of elements $ (E_1,a_1) \rightarrow (E_2,a_2)$. This ensures, by cancellation of pullbacks, that in the following diagram the left face of the prism is also a pullback, ensuring in turn that $ l_f^*\hirayo_{[(E_1,a_1,h_1)]_\sim} \simeq l_f^*\hirayo_{[(E_2,a_2,h_2)]_\sim}$
\[\begin{tikzcd}
	&& {f^*D_1} \\
	{a_2^*f^*D_2} && {} & {f^*D_2} \\
	&& {f^*E_1} \\
	{1_\mathcal{F}} &&& {f^*E_2}
	\arrow[""{name=0, anchor=center, inner sep=0}, "{f^*h_2}", from=2-4, to=4-4]
	\arrow["{f^*u}"{description}, from=3-3, to=4-4]
	\arrow[from=1-3, to=2-4]
	\arrow[from=2-1, to=1-3]
	\arrow["{a^*_2f^*h_2}"{description}, from=2-1, to=4-1]
	\arrow["{a_1}"{description}, from=4-1, to=3-3]
	\arrow["{a_2}"{description}, from=4-1, to=4-4]
	\arrow["\lrcorner"{anchor=center, pos=0.125}, draw=none, from=2-1, to=4-4]
	\arrow[from=2-1, to=2-4]
	\arrow["{f^*h_1}"', shorten <=-2pt, from=2-3, to=3-3]
	\arrow[no head, shorten >=2pt, from=1-3, to=2-3]
	\arrow["\lrcorner"{anchor=center, pos=0.125}, draw=none, from=1-3, to=0]
\end{tikzcd}\]

Before deducing from those observations that $ l_f^*$ is terminally connected, that is, lifts uniquely global elements, we should first give a word on what global elements in the pseudocolimit are -- we shall see that those coming from the pseudocolimit shall be enough to lift global element of arbitrary object in the bilimit topos. The terminal object of $\oplaxcolim_{1_\mathcal{F} \downarrow f^* \mathcal{E}/E}\mathcal{E}/E$ is the pair $(1_\mathcal{E}, \id_{1_\mathcal{F}})$ and the terminal map of any $(E,a,h)$ has as underlying morphism of elements the following triangle together with the terminal square as second component
\[\begin{tikzcd}
	& {f^*(E)} \\
	{1_\mathcal{F}} && {f^*(1_\mathcal{E})}
	\arrow["{\id_{1_\mathcal{F}}}"', Rightarrow, no head, from=2-1, to=2-3]
	\arrow["a", from=2-1, to=1-2]
	\arrow["{f^*(!_E)}", from=1-2, to=2-3]
\end{tikzcd} \hskip1cm
\begin{tikzcd}
	D & {1_\mathcal{E}} \\
	E & {1_\mathcal{E}}
	\arrow["{!_E}"', from=2-1, to=2-2]
	\arrow[Rightarrow, no head, from=1-2, to=2-2]
	\arrow["h"', from=1-1, to=2-1]
	\arrow["{!_D}", from=1-1, to=1-2]
\end{tikzcd}\]

Now, recall that cartesian morphisms in the oplax colimit enjoy a right calculus of fraction, so that a morphism $1_{\pscolim_{1_\mathcal{F} \downarrow f^*} \mathcal{E}/E} \rightarrow [(E,a,h)]_\sim $ in the localization $ \pscolim_{1_\mathcal{F} \downarrow f^*} \mathcal{E}/E$ can be presented as a span 
\[\begin{tikzcd}
	{(D,b, \id_D)} && {(E,a,h)} \\
	{(1_\mathcal{E},\id_{1_\mathcal{F}}, \id_{1_\mathcal{E}})}
	\arrow["{(u,h)}", from=1-1, to=1-3]
	\arrow["{(!_{D}, !_D)}"', from=1-1, to=2-1]
\end{tikzcd}\]
where the vertical leg is cartesian, which expresses the fact that any object of the form $ (D,b, \id_D)$ provides a representing object for the terminal object as pullbacks send identity maps on identity maps. In other words, any arrow $(u,h)$ as above shall provide a global element of $[(E,a,h)]_\sim$ in the pseudocolimit, whatever $ (D,b)$ was chosen as.

We can prove now $ l_f$ to be terminally connected. The inverse image part of $l_f : \mathcal{F} \rightarrow \bilim_{1_\mathcal{F} \downarrow f^*} \mathcal{E}/E$ is induced from the pullback functor $ l_f^* : \pscolim_{1_\mathcal{F}\downarrow f^*} \mathcal{E}/E \rightarrow \mathcal{F}$ through the left Kan extension. Also, as $l_f^*$ preserves colimit and any object in $ \Sh(\pscolim_{1_\mathcal{F}\downarrow f^*} \mathcal{E}/E, \bigvee_{1_\mathcal{F}\downarrow f^*} q_{(E,a, h)} (J_E))$ is a colimit of objects from the presentation site, the action of $l_f^*$ on arbitrary objects can be computed as a colimit of inverse images of basic objects. Concretely, if $ X \simeq \colim_{I_X} \hirayo_{[(E_i, a_i, h_i)]_\sim}$ in $ \bilim_{1_\mathcal{F} \downarrow f^*} \mathcal{E}/E $, then as $f^*$ preserves colimits, we have 
\[ l_f^*X \simeq \underset{I_X}\colim \; l_f^*\hirayo_{[(E_i, a_i, h_i)]_\sim}  \]
But each $ l_f^*\hirayo_{[(E_i, a_i, h_i)]_\sim}$ is computed as the pullback  
\[\begin{tikzcd}
	{l_f^*\hirayo_{[(E_i, a_i, h_i)]_\sim}} & {f^*D_i} \\
	{1_\mathcal{F}} & {f^*E_i}
	\arrow[from=1-1, to=2-1]
	\arrow[from=1-1, to=1-2]
	\arrow["{f^*h_i}", from=1-2, to=2-2]
	\arrow["{a_i}"', from=2-1, to=2-2]
	\arrow["\lrcorner"{anchor=center, pos=0.125}, draw=none, from=1-1, to=2-2]
\end{tikzcd}\]
As colimits are pullback stable in Grothendieck topoi, the following square is a pullback
\[\begin{tikzcd}
	{\underset{I_X}\colim \; l_f^*\hirayo_{[(E_i, a_i, h_i)]_\sim}} & {\underset{I_X}\colim \; f^*D_i } \\
	{\underset{I_X}\colim \; 1_\mathcal{F}} & {\underset{I_X}\colim \; f^*E_i}
	\arrow["{\underset{I_X}\colim \; a_i^*f^*h_i}"', from=1-1, to=2-1]
	\arrow[from=1-1, to=1-2]
	\arrow["{\underset{I_X}\colim \; f^*h_i}", from=1-2, to=2-2]
	\arrow["{\underset{I_X}\colim \; a_i}"', from=2-1, to=2-2]
	\arrow["\lrcorner"{anchor=center, pos=0.125}, draw=none, from=1-1, to=2-2]
\end{tikzcd}\]
As $f^*$ preserves colimits, this pullback expresses nothing but the fact that $l_f^*$ behaves on objects as it does on basic elements, that is, by taking fibers:
\[\begin{tikzcd}
	{l_f^*X} & {f^*\underset{I_X}\colim \; D_i} \\
	{\underset{I_X}\colim \; 1_\mathcal{F}} & {f^*\underset{I_X}\colim \; E_i}
	\arrow["{\pi_1}"', from=1-1, to=2-1]
	\arrow["\pi_2", from=1-1, to=1-2]
	\arrow["{f^*\underset{I_X}\colim \; h_i}", from=1-2, to=2-2]
	\arrow["{\underset{I_X}\colim \; a_i}"', from=2-1, to=2-2]
	\arrow["\lrcorner"{anchor=center, pos=0.125}, draw=none, from=1-1, to=2-2]
\end{tikzcd}\]

But then, by the universal property of the pullback, a global element $b : 1_\mathcal{F} \rightarrow l_f^*X$ of $l_f^*$ induces a morphism between global elements of $f^*$: 
\[\begin{tikzcd}
	&& {f^*\underset{I_X}\colim \; D_i } \\
	{1_\mathcal{F}} && {f^*\underset{I_X}\colim \; E_i}
	\arrow["{f^*\underset{I_X}\colim \; h_i}", from=1-3, to=2-3]
	\arrow["{\pi_2b}", from=2-1, to=1-3]
	\arrow["{(\underset{I_X}\colim \; a_i)\pi_1b}"', from=2-1, to=2-3]
\end{tikzcd}\]

which defines itself a morphism in the oplax colimit  
\[\begin{tikzcd}[sep=large]
	{(\underset{I_X}\colim \; D_i, \pi_2b, \id_{\underset{I_X}\colim \; D_i})} && {(\underset{I_X}\colim \;E_i, (\underset{I_X}\colim\; a_i)\pi_1b, \underset{I_X}\colim \; h_i)}
	\arrow["{(\underset{I_X}\colim\; h_i, \id_{\underset{I_X}\colim \; D_i})}", from=1-1, to=1-3]
\end{tikzcd}\]
The codomain of this morphism is conveniently sent to $ 1_\mathcal{F}$ by $l_f^*$ as seen in the diagram
\[\begin{tikzcd}
	{1_\mathcal{F}} & {f^*\underset{I_X}\colim \; D_i} \\
	{1_\mathcal{F}} & {f^*\underset{I_X}\colim \; D_i}
	\arrow["{\pi_2b}", from=2-1, to=2-2]
	\arrow[Rightarrow, no head, from=1-2, to=2-2]
	\arrow[Rightarrow, no head, from=1-1, to=2-1]
	\arrow["{\pi_2b}", from=1-1, to=1-2]
	\arrow["\lrcorner"{anchor=center, pos=0.125}, draw=none, from=1-1, to=2-2]
\end{tikzcd}\]
while $b$ is recovered as the morphism induced between the image as depicted below
\[\begin{tikzcd}[row sep=large]
	{1_\mathcal{F}} & {l_f^*X} \\
	{1_\mathcal{F}} & {f^*\underset{I_X}\colim \; D_i} & {f^*\underset{I_X}\colim \; D_i} \\
	&& {f^*\underset{I_X}\colim \;E_i}
	\arrow["{\pi_2b}", from=2-1, to=2-2]
	\arrow[Rightarrow, no head, from=1-1, to=2-1]
	\arrow["{f^*\underset{I_X}\colim\; h_i}"{description}, from=2-2, to=3-3]
	\arrow["{f^*\underset{I_X}\colim \; h_i}", from=2-3, to=3-3]
	\arrow["{(\underset{I_X}\colim\; a_i)\pi_1b}"'{pos=0.3}, curve={height=18pt}, from=2-1, to=3-3]
	\arrow[from=1-2, to=2-2]
	\arrow[from=1-2, to=2-3]
	\arrow["\lrcorner"{anchor=center, pos=0.125}, draw=none, from=1-2, to=3-3]
	\arrow["b", from=1-1, to=1-2]
\end{tikzcd}\]
In other words, $l_f^*$ lifts uniquely global elements of arbitrary objects, and $l_f$ is hence terminally connected. As this was proven without specification of a site of presentation for $ \mathcal{E}$, the factorization is ensured not to depend on a specific presentation.
\end{proof}

\subsection{Laxness conditions}\label{laxness conditions}

In general, for a given bifactorization system, the middle terms of the factorizations of the domain and codomain 1-cells of a non-invertible globular 2-cell may not be canonically related in any way: some factorization systems indeed lack such ``lax functoriality". However, it happens that the (terminally-connected, pro-etale) factorization system does enjoy such a structure.\\ 

Let be a globular 2-cell in $\GTop$ as below
\[\begin{tikzcd}
	{\mathcal{F}} && {\mathcal{E}}
	\arrow[""{name=0, anchor=center, inner sep=0}, "f", curve={height=-12pt}, from=1-1, to=1-3]
	\arrow[""{name=1, anchor=center, inner sep=0}, "g"', curve={height=12pt}, from=1-1, to=1-3]
	\arrow["\phi", shorten <=3pt, shorten >=3pt, Rightarrow, from=0, to=1]
\end{tikzcd}\]
It corresponds to a natural transformation $ \phi^\flat : f^* \Rightarrow g^*$. At the level of global elements, this natural transformation induces a functor 
\[\begin{tikzcd}
	{1_\mathcal{F} \downarrow f^*} & {1_\mathcal{F}\downarrow g^*}
	\arrow["{\phi^*}", from=1-1, to=1-2]
\end{tikzcd}\]
sending a global element $ a : 1 \rightarrow f^*E$ to the composite 
\[\begin{tikzcd}
	{1_\mathcal{F}} & {f^*E} \\
	& {g^*E}
	\arrow["a", from=1-1, to=1-2]
	\arrow["{\phi^\flat_E}", from=1-2, to=2-2]
	\arrow["{\phi^\flat_Ea}"', from=1-1, to=2-2]
\end{tikzcd}\]
This induces a 2-cell between the names of the corresponding elements
\[\begin{tikzcd}
	{\mathcal{F}} && {\mathcal{E}/E}
	\arrow[""{name=0, anchor=center, inner sep=0}, "a", bend left=20, start anchor=40, from=1-1, to=1-3]
	\arrow[""{name=1, anchor=center, inner sep=0}, "{\phi^*a}"', bend right=20, start anchor=-40, from=1-1, to=1-3]
	\arrow["{\phi^a}", shorten <=3pt, shorten >=3pt, Rightarrow, from=0, to=1]
\end{tikzcd}\]
whose inverse image component at an object $ h : D \rightarrow E$ in $\mathcal{E}/E$ is given by is the comparison map between the corresponding fibers 
\[\begin{tikzcd}
	{(\phi^\flat_Ea)^*g^*D} \\
	& {a^*f^*D} & {f^*D} & {g^*D} \\
	& {1_\mathcal{F}} & {f^*E} & {g^*E}
	\arrow["a"', from=3-2, to=3-3]
	\arrow["{\phi^\flat_E}"', from=3-3, to=3-4]
	\arrow["{g^*h}", from=2-4, to=3-4]
	\arrow["{f^*h}"{description}, from=2-3, to=3-3]
	\arrow["{\phi^\flat_D}", from=2-3, to=2-4]
	\arrow["{a^*f^*h}"{description}, from=2-2, to=3-2]
	\arrow[from=2-2, to=2-3]
	\arrow["\lrcorner"{anchor=center, pos=0.125}, draw=none, from=2-2, to=3-3]
	\arrow["{(\phi^\flat_Ea)^*g^*h}"', curve={height=12pt}, from=1-1, to=3-2]
	\arrow[curve={height=-18pt}, from=1-1, to=2-4]
	\arrow[dashed, from=2-2, to=1-1]
\end{tikzcd}\]
This 2-cell actually provides a decomposition of $\phi$ as the following pasting 
\[\begin{tikzcd}
	{\mathcal{F}} && {\mathcal{E}}
	\arrow[""{name=0, anchor=center, inner sep=0}, "f", curve={height=-12pt}, from=1-1, to=1-3]
	\arrow[""{name=1, anchor=center, inner sep=0}, "g"', curve={height=12pt}, from=1-1, to=1-3]
	\arrow["\phi", shorten <=3pt, shorten >=3pt, Rightarrow, from=0, to=1]
\end{tikzcd} =\begin{tikzcd}
	{\mathcal{F}} && {\mathcal{E}/E} & {\mathcal{E}}
	\arrow[""{name=0, anchor=center, inner sep=0}, "a"{description}, curve={height=-12pt}, from=1-1, to=1-3]
	\arrow[""{name=1, anchor=center, inner sep=0}, "{\phi^*a}"{description}, curve={height=12pt}, from=1-1, to=1-3]
	\arrow[from=1-3, to=1-4]
	\arrow[""{name=2, anchor=center, inner sep=0}, "f", curve={height=-24pt}, from=1-1, to=1-4]
	\arrow[""{name=3, anchor=center, inner sep=0}, "g"', curve={height=24pt}, from=1-1, to=1-4]
	\arrow["{\phi^a}", shorten <=3pt, shorten >=3pt, Rightarrow, from=0, to=1]
	\arrow["\simeq"{description}, draw=none, from=2, to=1-3]
	\arrow["\simeq"{description}, draw=none, from=3, to=1-3]
\end{tikzcd}\]

Let us see how the pseudolimit formula for the comprehensive factorization interacts with those data. On the pro-etale side, the functor $ \phi^*$ defines a transformation of pseudocones
\[\begin{tikzcd}
	{1_\mathcal{F} \downarrow f^*} && {1_\mathcal{F}\downarrow g^*} \\
	& {\GTop/\mathcal{E}}
	\arrow[from=1-1, to=2-2]
	\arrow[from=1-3, to=2-2]
	\arrow[""{name=0, anchor=center, inner sep=0}, "{\phi^*}", from=1-1, to=1-3]
	\arrow["\simeq"{description}, draw=none, from=0, to=2-2]
\end{tikzcd}\]
inducing a geometric morphism between the pseudolimits 
\[\begin{tikzcd}
	{\underset{(E,a) \in 1_\mathcal{F} \downarrow f^*}\bilim \; \mathcal{E}/E} && {\underset{(E,b) \in 1_\mathcal{F} \downarrow g^*}\bilim \; \mathcal{E}/E} \\
	& {\mathcal{E}}
	\arrow["{\pi_{g}}", from=1-3, to=2-2]
	\arrow["{\pi_f}"', from=1-1, to=2-2]
	\arrow[""{name=0, anchor=center, inner sep=0}, "\widetilde\phi"', from=1-3, to=1-1]
	\arrow["\simeq"{description}, draw=none, from=0, to=2-2]
\end{tikzcd}\]
where the morphism $ \widetilde{\phi} $ is induced from the morphism between the (large) pseudocolimits of sites
\[\begin{tikzcd}
	{\underset{(E,a) \in 1_\mathcal{F} \downarrow f^*}\pscolim \; \mathcal{E}/E} && {\underset{(E,b) \in 1_\mathcal{F} \downarrow g^*}\pscolim \; \mathcal{E}/E}
	\arrow["\overline\phi", from=1-1, to=1-3]
\end{tikzcd}\]
sending an object represented by a triple $((E,a), h) $ to the object represented by the triple $ (E, \phi^\flat_Ea), h)$. \\

On the terminally connected side, the data of the $ (\phi^a)_{a \in 1_\mathcal{F} \downarrow f^*}$, as defined above from the pullback property, induce together with $\overline{\phi}$ a canonical 2-cell
\[\begin{tikzcd}
	& {\mathcal{F}} \\
	{\underset{(E,a) \in 1_\mathcal{F} \downarrow f^*}\pscolim \; \mathcal{E}/E} && {\underset{(E,b) \in 1_\mathcal{F} \downarrow g^*}\pscolim \; \mathcal{E}/E}
	\arrow["\overline\phi"', from=2-1, to=2-3]
	\arrow[""{name=0, anchor=center, inner sep=0}, "{\langle b^*\rangle_{(E,b) \in 1_\mathcal{F} \downarrow g^* } }"', from=2-3, to=1-2]
	\arrow[""{name=1, anchor=center, inner sep=0}, "{\langle a^*\rangle_{(E,a) \in 1_\mathcal{F} \downarrow f^* } }", from=2-1, to=1-2]
	\arrow["{(\phi^a)_{(E,b) \in 1_\mathcal{F} \downarrow g^* } }", shift right=4, shorten <=8pt, shorten >=8pt, Rightarrow, from=1, to=0]
\end{tikzcd}\]
whose associated geometric transformation can be notated $\lambda_\phi$. Combining those data provides us with a canonical decomposition of the globular 2-cell $ \phi$:
\[\begin{tikzcd}
	& {\underset{(E,a) \in 1_\mathcal{F} \downarrow f^*}\bilim \; \mathcal{E}/E} \\
	{\mathcal{F}} && {\mathcal{E}} \\
	& {\underset{(E,b) \in 1_\mathcal{F} \downarrow g^*}\bilim \; \mathcal{E}/E}
	\arrow[""{name=0, anchor=center, inner sep=0}, "\widetilde\phi"{description}, from=3-2, to=1-2]
	\arrow[""{name=1, anchor=center, inner sep=0}, "{t_g}"', end anchor=162, from=2-1, to=3-2]
	\arrow[""{name=2, anchor=center, inner sep=0}, "{t_f}", from=2-1, to=1-2]
	\arrow["{\pi_g}"', start anchor=20, from=3-2, to=2-3]
	\arrow["{\pi_f}", from=1-2, to=2-3]
	\arrow["{\lambda_\phi}", shift left=2, shorten <=4pt, shorten >=4pt, Rightarrow, from=2, to=1]
	\arrow["\simeq"{description}, draw=none, from=0, to=2-3]
\end{tikzcd}\]

Hence, we can say that the (terminally connected, pro-etale) factorization system enjoys \emph{oplax factorization of globular cells}, a notion which shall be developed in a future work. As a consequence, we can deduce a stronger orthogonality property of the terminally connected morphismes left to pro-etale: 

\begin{proposition}
    For any pair $ t$, $p$ with $ t$ terminally connected and $p$ pro-etale, any oplax-square as below
\[\begin{tikzcd}
	{\mathcal{G}} & {\mathcal{H}} \\
	{\mathcal{F}} & {\mathcal{E}}
	\arrow[""{name=0, anchor=center, inner sep=0}, "t"', from=1-1, to=2-1]
	\arrow["g", from=1-1, to=1-2]
	\arrow[""{name=1, pos=0.525, inner sep=0}, "p", from=1-2, to=2-2]
	\arrow["f"', from=2-1, to=2-2]
	\arrow["\phi"', shorten <=6pt, shorten >=6pt, Rightarrow, from=1, to=0]
\end{tikzcd}\]
admits a canonical decomposition of the following form
\[\begin{tikzcd}[sep=large]
	{\mathcal{G}} & {\mathcal{H}} \\
	{\mathcal{F}} & {\mathcal{E}}
	\arrow[""{name=0, pos=0.51, inner sep=0}, "t"', from=1-1, to=2-1]
	\arrow["g", from=1-1, to=1-2]
	\arrow[""{name=1, pos=0.525, inner sep=0}, "p", from=1-2, to=2-2]
	\arrow["f"', from=2-1, to=2-2]
	\arrow[""{name=2, anchor=center, inner sep=0}, "{h_\phi}"{description}, from=2-1, to=1-2]
	\arrow["{\lambda_\phi}"', shift right=2, shorten <=3pt, Rightarrow, from=2, to=0]
	\arrow["{\rho_\phi \atop \simeq}", shift left=2, shorten <=3pt, shorten >=3pt, draw=none, from=1, to=2]
\end{tikzcd}\]
with an invertible 2-cell on the right.
\end{proposition}

\begin{proof}
    Take the oplax factorization of $\phi$ seen as a globular cell $ pg \Rightarrow ft$; as precomposing with a terminally connected geometric morphism (respectively, post-composing with a pro-etale geometric morphism) does not modify the factorization, we get a decomposition of $\phi$ as the following pasting 
\[\begin{tikzcd}[sep=small]
	& {\underset{1_\mathcal{G} \downarrow g^*}\bilim \, \mathcal{H}/H} & {\mathcal{H}} \\
	{\mathcal{G}} &&& {\mathcal{E}} \\
	& {\mathcal{F}} & {\underset{1_\mathcal{F} \downarrow f^*}\bilim \, \mathcal{E}/E}
	\arrow["t"', from=2-1, to=3-2]
	\arrow["p", from=1-3, to=2-4]
	\arrow["{t_f}"', from=3-2, to=3-3]
	\arrow["{t_g}", from=2-1, to=1-2]
	\arrow["{\pi_g}", from=1-2, to=1-3]
	\arrow["{\pi_f}"', from=3-3, to=2-4]
	\arrow[""{name=0, anchor=center, inner sep=0}, "{\widetilde{\phi}}"{description}, from=3-3, to=1-2]
	\arrow["\simeq"{description}, draw=none, from=1-3, to=3-3]
	\arrow["{\lambda_\phi}"', shorten <=19pt, shorten >=19pt, Rightarrow, from=0, to=2-1]
\end{tikzcd}\]
\end{proof}

\begin{remark}
It is not a coincidence that the other main example of 2-dimensional factorization system enjoying such a lax orthogonality property is the comprehensive factorization system on $\Cat$. We are discussing below some conjectures about this similarity. Such laxness property also often comes with additional stability condition along bicomma and bicocomma squares:
\end{remark}

\begin{proposition}\label{terco are stable along bicocomma}
Terminally connected geometric morphisms are stable left to bicocomma: that is, in a bicocomma square as below in $\GTop$
\[\begin{tikzcd}
	{\mathcal{F}} & {\mathcal{G}} \\
	{\mathcal{E}} & {f \uparrow t}
	\arrow[""{name=0, anchor=center, inner sep=0}, "t"', from=1-1, to=2-1]
	\arrow["f", from=1-1, to=1-2]
	\arrow[""{name=1, anchor=center, inner sep=0}, "{p_1}", from=1-2, to=2-2]
	\arrow["{p_2}"', from=2-1, to=2-2]
	\arrow["{\phi_{f,t}}"', shorten <=6pt, shorten >=6pt, Rightarrow, from=1, to=0]
\end{tikzcd}\]
if $ t$ is terminally connected, then so is $p_1$. 
\end{proposition}

\begin{proof}
    Recall that the bicocomma square $f\uparrow t$ in $\GTop$ has as underlying category the bicomma square $ f^*\uparrow t^* $ in $\Cat$ and is hence equiped with a universal 2-cell
\[\begin{tikzcd}
	{\mathcal{F}} & {\mathcal{G}} \\
	{\mathcal{E}} & {f^* \downarrow t^*}
	\arrow["{f^*}"', from=1-2, to=1-1]
	\arrow[""{name=0, anchor=center, inner sep=0}, "{p_1^*}"', from=2-2, to=1-2]
	\arrow["{p_2^*}", from=2-2, to=2-1]
	\arrow[""{name=1, anchor=center, inner sep=0}, "{t^*}", from=2-1, to=1-1]
	\arrow["{\phi^\flat_{f,t}}"', shorten <=7pt, shorten >=7pt, Rightarrow, from=0, to=1]
\end{tikzcd}\]
where the category $ f^*\downarrow t^*$ has as objects triples $ (G,E,u)$ with $ u : f^*G \rightarrow t^*E $ in $\mathcal{F}$, with the projection $p_1^*(G,E,u) = G$. The terminal object in $f^*\downarrow t^*$ is the triple $(1_\mathcal{G}, 1_{\mathcal{E}}, ! : f^*(1_\mathcal{G}) \simeq t^*(1\mathcal{E}))$; a global element in $ f^* \downarrow t^*$ of some $(G,E,u)$ is the same as a choice of $(a,b)$ inscribed in a triangle
\[\begin{tikzcd}
	& {f^*G} \\
	{1_\mathcal{F}} & {t^*E}
	\arrow["{f^*a}", from=2-1, to=1-2]
	\arrow["u", from=1-2, to=2-2]
	\arrow["{t^*b}"', from=2-1, to=2-2]
\end{tikzcd}\]
Now, a global element $a : 1_\mathcal{G} \rightarrow p^*(G, E, u) $ simply is an element $ a : 1_\mathcal{G} \rightarrow G $, and the composite $ uf^*a a : 1_\mathcal{F} \rightarrow t^*E$ defines a global element of $t^*$, which lifts uniquely to a global element $ b : 1_\mathcal{E} \rightarrow E$ for $ t$ is terminally connected. Hence, our global element for $p_1^*$ lifts uniquely as this $(a, b) $ in $f^*\downarrow t^*$; one can check this choice to be moreover natural in $a$. 
\end{proof}

We are seeing also in section 5 that terminally connected geometric morphisms also enjoy a dual bicomma stability property, but this is not directly related to the lax orthogonality property. 

%% file: Sections/Section_4.tex
\section{Properties of pro-etale geometric morphisms}

Pro-etale geometric morphisms have been known for some time, though not much beyond their definition; few information is available on them, and they are still lacking an intrinsic characterization. Here, we would like to provide a first systematic investigation of their properties. 

\subsection{Pro-adjoint aspects}

 We can already guess that pro-etale geometric morphisms, generalizing etaleness out of the locally connected world, shall lack the further left adjoint to the inverse image part of etale geometric morphisms, which expresses the fact that their connected components do not form anymore a set, but rather a prodiscrete space. 

\begin{proposition}
A geometric morphism is etale if and only if it is pro-etale and essential.
\end{proposition}

\begin{proof}
Suppose that $ p : \mathcal{F} \rightarrow \mathcal{E}$ is pro-etale and essential. Then $f^*$ has a further left adjoint $f_!$ so that the category of global elements $ 1_\mathcal{F} \downarrow f^*$ has an initial object $ \eta : 1_\mathcal{F} \rightarrow f^*f_!1_\mathcal{F}$. But we saw at \cref{canonical presentation of proetale} that a pro-etale was always decomposable as the cofiltered bilimit over $ 1_\mathcal{F}\downarrow p^*$, which by initialness reduces to the etale geometric morphism over $ f_!(1_\mathcal{F})$, hence a geometric equivalence over $\mathcal{E}$
\[\begin{tikzcd}
	{\mathcal{F}} && {\mathcal{E}/f_!1_\mathcal{F}} \\
	& {\mathcal{E}}
	\arrow["\simeq", from=1-1, to=1-3]
	\arrow[""{name=0, anchor=center, inner sep=0}, "p"', from=1-1, to=2-2]
	\arrow[""{name=1, anchor=center, inner sep=0}, "{\pi_{f_!1_\mathcal{F}}}", from=1-3, to=2-2]
	\arrow["\simeq"', draw=none, from=1, to=0]
\end{tikzcd}\]
\end{proof}

\begin{proposition}
Any pro-etale geometric morphism is localic. 
\end{proposition}

\begin{proof}
Etale geometric morphisms are localic, corresponding to the discrete internal locales at objects. But now localic geometric morphisms are closed under cofiltered bilimits for they are a right class in an orthogonality structure.
\end{proof}


\begin{remark}
Hence, any pro-etale geometric morphism over $\mathcal{E}$ defines uniquely an internal locale in $\Loc[\mathcal{E}]$, which is the cofiltered limit of the corresponding discrete internal locales - computed in $\Loc[\mathcal{E}]$. In other words, any pro-etale geometric morphism is the localic geometric morphism at a pro-discrete internal locale. Recall that the category of internal locales in $\mathcal{E}$ is equivalent to the category of localic geometric morphisms over $\mathcal{E}$, and moreover that $ \Loc/\mathcal{E} \hookrightarrow \GTop/\mathcal{E} $ is reflective, with the left adjoint given by the localic reflection: hence the functor sending an internal locale $ L$ to $ \Sh_\mathcal{E}(L) \rightarrow \mathcal{E}$ preserves limits, in particular cofiltered limits. Hence, for a pro-etale geometric morphism given by a cofiltered diagram $ E : I \rightarrow \mathcal{E}$, we have 
\[  \underset{i \in I}{\bilim} \, \mathcal{E}/E_i \simeq \Sh_{\mathcal{E}}(\underset{i \in I}{\lim} \, E_i)  \]
where $ \lim_{i \in I}E_i$ is computed in the category of internal locales $ \Loc[\mathcal{E}]$: any pro-etale geometric morphism is the localic morphism associated to a pro-discrete internal locale. 
\end{remark}

\begin{remark}
We saw at \cref{proetale are discrete} that pro-etale geometric morphisms over $ \mathcal{E}$ are still discrete as objects of the 2-category $\GTop/\mathcal{E}$, in the sense that their representables return homcategories that are equivalent to sets (rather set themselves). In a sense, this definition of discreteness is a bit coarse and would rather correspond to being \emph{groupoidal} from the point of view of the strict slice, which explains the discrepancy with the fact hat the corresponding internal locales are no more discrete. 
\end{remark}

As we see, several properties of pro-etales geometric morphisms are reformulation of properties of etale geometric morphisms in the absence of the further left adjoint to the inverse image. However, although this further left adjoint does not exist globally, it is still present in a hidden way, at the level of the \emph{pro-completion}. There is, since \cite{tholen_1984}, a general theory of generalized adjunctions where a left adjoint only exists at the level of a free completion for some shape of diagrams. In particular is investigated the example of \emph{right pro-adjoints}. 

\begin{definition}\label{free procompletion}
    Let $\mathcal{A}$ be a category. A \emph{pro-object} in $ \mathcal{A}$ is a functor $X : I \rightarrow \mathcal{A}$ with $ I$ cofiltered. A morphism of pro-objects $ (I,X) \rightarrow (J,Y)$ is defined as follows; first, choose the data of a reindexing map $f: J \rightarrow I$ together with a transformation 
\[\begin{tikzcd}
	I \\
	& {\mathcal{A}} \\
	J
	\arrow["f", from=3-1, to=1-1]
	\arrow[""{name=0, anchor=center, inner sep=0}, "X", from=1-1, to=2-2]
	\arrow[""{name=1, anchor=center, inner sep=0}, "Y"', from=3-1, to=2-2]
	\arrow["\phi"', shorten <=4pt, shorten >=4pt, Rightarrow, from=0, to=1]
\end{tikzcd}\]
which is not strictly speaking natural but natural up to the following relation: for any $ d : j \rightarrow j'$ in $J$, there exists a span $u : i \rightarrow f(j)$, $u' : i \rightarrow f(j')$ such that we have a commutation 
\[\begin{tikzcd}[sep=small]
	& {X(f(j))} & {Y(j)} \\
	{X(i)} \\
	& {X(f(j'))} & {Y(j')}
	\arrow["{X(u')}"', from=2-1, to=3-2]
	\arrow["{X(u)}", from=2-1, to=1-2]
	\arrow["{\phi_j}", from=1-2, to=1-3]
	\arrow["{\phi_{j'}}"', from=3-2, to=3-3]
	\arrow["{Y(d)}", from=1-3, to=3-3]
\end{tikzcd}\]
Then, identify any two choices $ (f,\phi)$, $(f', \phi')$ if for any $j$ in $J $, there exists a span $ d:i \rightarrow f(j)$, $d': i \rightarrow f'(j)$ and making the following diagram to commute
\[\begin{tikzcd}[sep=small]
	& {X(f(j))} \\
	{X(i)} && {Y(j)} \\
	& {X(f'(j))}
	\arrow["{\phi_j}", from=1-2, to=2-3]
	\arrow["{\phi_{j'}}"', from=3-2, to=2-3]
	\arrow["{X(d)}", from=2-1, to=1-2]
	\arrow["{X(d')}"', from=2-1, to=3-2]
\end{tikzcd}\]
Finally, a morphism $ X \rightarrow Y$ in $\Pro(\mathcal{A})$ is a class for this relation. In the following, we denote as $ \Pro(\mathcal{A})$ the category of pro-objects and morphisms of pro-objects between them. For any locally small category $ \mathcal{A}$, there is an full and faithful functor $ \iota_\mathcal{A} : \mathcal{A} \hookrightarrow \Pro(\mathcal{A})$ sending $ A$ on the singleton pro-object $ A : * \rightarrow \mathcal{A}$. 
\end{definition}

\begin{remark}
The category $ \Pro(\mathcal{A})$ can also be described as the full subcategory of $ [\mathcal{A}, \Set]^{\op}$ whose objects are those functors that are filtered colimits of corepresentables. This amounts to their category of elements being cofiltered. 
\end{remark}

\begin{remark}\label{terminal proobject}
If $ \mathcal{A}$ has a terminal object, then so does $\Pro(\mathcal{A})$, and one can check it coincides with the image $ \iota_A(1_\mathcal{A})$, with the terminal map of a pro-object being given by the class of any such triangle as below:
\[\begin{tikzcd}[sep=large]
	I \\
	{*} & {\mathcal{A}}
	\arrow[""{name=0, anchor=center, inner sep=0}, "X", from=1-1, to=2-2]
	\arrow["i", from=2-1, to=1-1]
	\arrow[""{name=1, anchor=center, inner sep=0}, "{1_\mathcal{A}}"', from=2-1, to=2-2]
	\arrow["{!_{X(i)}}"', shorten <=2pt, shorten >=2pt, Rightarrow, from=0, to=1]
\end{tikzcd}\]
Indeed one can show that any two choice of $i$ in $I$ are identified through the relation above.
\end{remark}

\begin{definition}
    A functor $ f : \mathcal{A} \rightarrow \mathcal{B}$ is said to be \emph{right pro-adjoint} if it admits a relative left adjoint $ l$ along the embedding $ \iota_\mathcal{A} $
\[\begin{tikzcd}
	{\Pro(\mathcal{A})} \\
	{\mathcal{A}} & {\mathcal{B}}
	\arrow["{\iota_\mathcal{A}}", hook, from=2-1, to=1-1]
	\arrow["f"', from=2-1, to=2-2]
	\arrow[""{name=0, anchor=center, inner sep=0}, "l"', from=2-2, to=1-1]
	\arrow["\epsilon"', shorten <=3pt, Rightarrow, from=0, to=2-1]
\end{tikzcd}\]
which means that for any $ A$ in $\mathcal{A}$ and $ B$ in $\mathcal{B}$ we have a natural isomorphism 
\[  \mathcal{B}[B, f(A)] \simeq \Pro(\mathcal{A})[l(B), \iota_\mathcal{A}(A)]  \]
\end{definition}

It is known that any left exact functor is right pro-adjoint: this is a consequence that its associated nerve is a filtered colimit of representable, so that its category of elements is cofiltered. As the consequence we have the following known fact:

\begin{proposition}
Let $ f : \mathcal{F} \rightarrow \mathcal{E}$ be any geometric morphism. Then $f^*$ is right pro-adjoint. 
\end{proposition}

\begin{proof}
One can check that the relative left adjoint $ f_!$ sends $ F$ to the projection $ \pi_F : F \downarrow f^* \rightarrow \mathcal{E}$ sending $a: F \rightarrow f^*E$ to $E$. For $ f^*$ is lex, hence representably flat, any such $F \downarrow f^*$ is cofiltered, and this is a pro-object. Similarly it sends $ u: F_1 \rightarrow F_2$ to (the class of) the strictly commutative triangle
\[\begin{tikzcd}
	{F_1 \downarrow f^*} \\
	& {\mathcal{C}} \\
	{F_2 \downarrow f^*}
	\arrow[""{name=0, anchor=center, inner sep=0}, "{\pi_{F_1}}", from=1-1, to=2-2]
	\arrow["{u\downarrow f^*}", from=3-1, to=1-1]
	\arrow[""{name=1, anchor=center, inner sep=0}, "{\pi_{F_2}}"', from=3-1, to=2-2]
	\arrow["{=}"', draw=none, from=0, to=1]
\end{tikzcd}\]
\end{proof}

\begin{remark}
Remark that for an essential geometric morphism, each comma category $ F\downarrow f^*$ has an initial object $\eta_F :F \rightarrow f^*f_!F$, and the relative left adjoint actually factorizes through $\mathcal{E}$ as a global left adjoint. 
\end{remark}

Some properties of etale geometric morphisms conveyed by their essential image part generalize to pro-etale geometric morphism as a property of the left pro-adjoint. For instance, recall from \cite{lurie2009higher}[Proposition 6.3.5.11] that the essential image $ E_!$ of an etale morphism at any object $E$ is faithful. By a manipulation of adjoints, we equivalently have that its inverse image $E^*$ is \emph{coseparating}, which means that two parallel arrows $ a,a' : g \rightrightarrows h$ are equals in a slice topos $\mathcal{E}/E$ if and only if they are jointly coequalized by all arrows of the form $h \rightarrow E_i^*E$.

\begin{proposition}
    If $f : \mathcal{F} \rightarrow \mathcal{E}$ is pro-etale, then the left pro-adjoint $ f_!$ is faithful.
\end{proposition}

\begin{proof}
As $ f$ is pro-etale, it is up to an equivalence of the form $ \mathcal{F} \simeq  \bilim_I \mathcal{E}/E_i \rightarrow \mathcal{E}$ for some cofiltered diagram $I \rightarrow \mathcal{E}$, with the inverse image $f^*$ sending an object $E$ to the object represented by any triple of the form $ (i, E_i^*E)$ in the pseudocolimit $ \pscolim_{I^\op} \mathcal{E}/E_i$. As objects of the pseudocolimit generate $\mathcal{F}$, it is sufficient to give an expression of the restriction of the left pro-adjoint $ f_!$ to the pseudocolimit, and to test faithfulness at this level. Let be a parallel pair $[(d,a)]_\sim,[(d',a')]_\sim : [(j,g)]_\sim \rightrightarrows [(i,h)]_\sim$ in $\pscolim_{I^\op} \mathcal{E}/E_i$, conssiting of a parallel pair $ d,d' : j \rightarrow i$ in $I$ and a span over $E_j$
\[\begin{tikzcd}
	& {u_d^*D} \\
	B && {u_{d'}^*D} \\
	& {E_j}
	\arrow["{u_d^*h}"{description, pos=0.3}, from=1-2, to=3-2]
	\arrow["a", from=2-1, to=1-2]
	\arrow["{a'}"{description, pos=0.7}, from=2-1, to=2-3]
	\arrow["g"', from=2-1, to=3-2]
	\arrow["{u_{d'}^*D}", from=2-3, to=3-2]
\end{tikzcd}\]
For $I$ is cofiltered, there exists an arrow $ d'' : k \rightarrow j$ such that $ d d'' = d'd''$. Then we end up with a parallel pair over $E_k$
\[\begin{tikzcd}
	{u_{d''}^*B} && {u_{dd''}^*D \simeq u_{d'd''}^*D} \\
	& {E_k}
	\arrow["{u_{d''}^*a}", shift left, from=1-1, to=1-3]
	\arrow["{u_{d''}^*a'}"', shift right, from=1-1, to=1-3]
	\arrow["{u_{d''}^*g}"', from=1-1, to=2-2]
	\arrow["{u_{dd''}^*D \simeq u_{d'd''}^*D}", from=1-3, to=2-2]
\end{tikzcd}\]
such that $ [(d,a)]_\sim = [(dd'', u_{dd''}^*a)]_\sim$ and $ [(d',a')]_\sim = [(d'd'', u_{d'd''}^*a')]_\sim$. In other words, it is always possible to find representing objects such that the parallel pair is supported by a same arrow in $ I$, so that it corresponds to a parallel pair in the slice over the domain object. Hence, we can assume without loss of generality that $d=d'$, so that $ u_d^*h = u_{d'}^*h$ in $\mathcal{E}/E_j$. 

Suppose now that $ f_![(d,a)]_\sim = f_![(d,a')]_\sim$, which means that the functors 
\[\begin{tikzcd}
	{[(i,h)]_\sim \downarrow f^*} && {[(j,g)]_\sim \downarrow f^*}
	\arrow["{[(d,a)]_\sim\downarrow f^*}", shift left, from=1-1, to=1-3]
	\arrow["{[(d,a')]_\sim\downarrow f^*}"', shift right, from=1-1, to=1-3]
\end{tikzcd}\]
are the same; objects of $[(i,h)]_\sim \downarrow f^*$ are span that can be chosen of the form
\[\begin{tikzcd}
	{(l, e^*h)} & {(l, E_{l}^*E)} \\
	{(i,h)}
	\arrow["{(1_l, b)}", from=1-1, to=1-2]
	\arrow["{(e,1_{e^*D}) }"', from=1-1, to=2-1]
\end{tikzcd}\]
Amongst them are those which arise more directly in $ \mathcal{E}/E_i$ itself from arrows $ b : h \rightarrow E_i^*E $. So the condition that $[(d,a)]_\sim\downarrow f^* = [(d,a')]_\sim\downarrow f^*$ implies in particular that those two functors agree on all such elements, which means that in $\mathcal{E}/E_i$ any composite $ ba$ and $ba'$ as below are equal as arrows $ u_dg \rightarrow E_i^*E$
\[\begin{tikzcd}
	B & D & {E_i\times E} \\
	{E_j} & {E_i}
	\arrow["a", shift left, from=1-1, to=1-2]
	\arrow["{a'}"', shift right, from=1-1, to=1-2]
	\arrow["g"', from=1-1, to=2-1]
	\arrow["b", from=1-2, to=1-3]
	\arrow["h"{description}, from=1-2, to=2-2]
	\arrow["{E_i^*E}", from=1-3, to=2-2]
	\arrow["{u_d}"', from=2-1, to=2-2]
\end{tikzcd}\]
But then this forces $ a=a'$ as the objects of the form $ E_i^*E$ for $E$ are a coseparator of $\mathcal{E}/E_i$. Hence, $(d,a) = (d',a')$, and faithfulness of the left pro-adjoint $f_!$. \end{proof}

 \subsection{Intrinsic characterization}
 
Let us now try to give a more intrinsic characterization of pro-etale geometric morphisms.

\begin{division}
 From the orthogonal factorization theorem, we know that a geometric morphism $f$ is pro-etale if and only if its terminally connected part $ t_f$ is a geometric equivalence, which amounts to the inverse image part $t_f^*$ to being an equivalence of categories, equivalently, that it is essentially surjective on objects and fully faithful.

But we saw in the proof of \cref{terminally connected-proetale factorization} that $t_f^*$ could be computed in an explicit way. First, for an object in the pseudocolimit $\pscolim_{1_\mathcal{F}\downarrow f^*} \mathcal{E}/E$ presented as the equivalence class of some triple $(E,a,h)$ with $ a:1\mathcal{F} \rightarrow f^*E$ and $ h : D \rightarrow E$, we had $ t_f^*([(E,a,h)]_\sim = a^*f^*D$ and composition of pullbacks ensures that this expression does not depend on the choice of representing objects in the equivalence class. Then, for an object $X$ in the topos of sheaves over this pseudocolimit of sites, constructed as a colimit of representables $ X \simeq \colim_{i \in I} \hirayo_{[(E_i,a_i,h_i)]_\sim}$ presented by a choice of family of global elements $ (a_i : 1_\mathcal{E} \rightarrow f^*E_i$ and arrows $ h_i : D_i \rightarrow E_i$, we had the formula
$ t_f^*(X) = \colim_{i \in I} a_i^* f^*D_i $.

Hence, asking for $ t_f^*$ to be an equivalence amounts to asking the following conditions:\begin{itemize}
    \item for any object $F$ in $\mathcal{F}$, there is a diagram $ (E_i, a_i, h_i)_{i \in I}$ with $ a_i : 1_\mathcal{F} \rightarrow f^*E_i$ and $ h_i : D_i \rightarrow E_i$, with the transition morphism at a $d : i \rightarrow j$ consisting in a pair $ (u_d, v_d) : h_i \rightarrow h_j$ forming a commutative square with the first component defining a morphism of elements $ u_d : a_i \rightarrow a_j$, such that one has a colimit decomposition into the corresponding fibers: 
\[ F \simeq \underset{i \in I}{\colim} \; a_i ^*f^*D_i \]
    \item for any morphism $ u : F \rightarrow F'$, there is a diagram of morphisms $ (u_i,v_i) : (E_i, a_i, h_i) \rightarrow (E'_i, a'_i, h'_i)$ such that $u$ is obtained as the colimit in the arrow category of the induced maps $ (u_i,v_i) : a_i^*f^*D_i \rightarrow (a_i')^*f^*D_i'$;
    \item moreover, the faithfulness condition says that two morphisms of diagrams that beget the same morphism between their colimit must actually be equals.
    \end{itemize} 
We shall say in this case that $f^*$ \emph{generates $ \mathcal{F}$ under fibers of global elements}.

\end{division}

\begin{proposition}\label{characterization of pro-etales}
    A geometric morphism is pro-etale if and only if its inverse image generates its domain under fibers of global elements.
\end{proposition}

\begin{proof}
From what precedes any pro-etale geometric morphisms generates its domain under fibers of global elements of its inverse image. Let us prove the reverse condition. Suppose that $ f : \mathcal{F} \rightarrow \mathcal{E}$ satisfies this property: we must then decompose $f$ as a cofiltered bilimit indexed by its diagram of global elements. In fact, we can foretell that one can choose as cofiltered indexing diagram the category of global elements $ 1_\mathcal{F} \downarrow f^*$. We have to prove that $ f$ decomposes in $\GTop/\mathcal{E}$ as the colfiltered bilimit $ \lim_{1_\mathcal{F} \downarrow f^*} \pi_E$, or in other words, that for any other geometric morphism $ g : \mathcal{G} \rightarrow \mathcal{E}$ one has an equivalence of homcategories 
    \[  \GTop/\mathcal{E}[g, f] \simeq \underset{(E,a) \in 1_\mathcal{F} \downarrow f^*}{\bilim} \GTop/\mathcal{E}[g, \pi_E] \]
where the latter bilimits corresponds to families of global elements of $(b_{(E,a)} : 1_\mathcal{G} \rightarrow g^*E)_{(E,a) \in 1_\mathcal{F} \downarrow f^*}$ indexed by global elements of $f^*$.
In one direction, it is always true that a geometric morphism $ h : \mathcal{G} \rightarrow \mathcal{F}$ with $fh \simeq g$ defines such a cone of global element $(\ulcorner a \urcorner h)_{(E,a) \in 1_\mathcal{F} \downarrow f^*}$ by postcomposition of $h$ with the names of the corresponding element of $f^*$. 

The converse direction makes use of the characterization. Suppose we have such a family $(b_{(E,a)} : 1_\mathcal{G} \rightarrow g^*E)_{(E,a) \in 1_\mathcal{F} \downarrow f^*}$ of global elements of $g^*$. We want to define a geometric morphism $ h : \mathcal{G} \rightarrow \mathcal{F}$ together with a natural transformation $ g \simeq fh$ such that each composite $ \ulcorner a \urcorner h : g \rightarrow \pi_{E}$ coincides with the name of the element $b_{(E,a)}$. Such a geometric morphism can be constructed through its inverse image as follows. As we suppose that $f$ generates $\mathcal{F}$ under fibers of global elements, we can choose for any $ F$ in $\mathcal{F}$ a colimit decomposition as above in term of fibers $ a_i^* f^*D_i$ where $ a_i : 1_\mathcal{F} \rightarrow f^*E_i$ are a certain family of global elements of $f^*$ and $ u_i : D_i \rightarrow E_i$ a family of morphisms in $\mathcal{E}$. Each of those global elements corresponds to a global element $ b_{(E_i, a_i)} : 1_\mathcal{G} \rightarrow g^*E_i$ in $\mathcal{G}$ given by the cone. Then, define $h^*F$ as the colimit of the corresponding fibers 
\[ h^*F \simeq \underset{i \in I}{\colim} \; b_{(E_i,a_i)} ^*g^*D_i \]
with $ b_{(E_i,a_i)} ^*g^*D_i$ being the fibers of $ g^*(u_i)$ along $b_{(E_i,a_i)}$. 

This construction is well defined, that is, does not depend on the choice of the colimit decomposition: suppose one has two such decompositions $F \simeq \colim_{i \in I} a_i^*f^*D_i \simeq \colim_{j\in J} b_j^*f^*H_j $ with respectively $ u_i : D_i \rightarrow E_i$, $v_ j : H_j \rightarrow E_j$ the maps we take fibers of. As the cone of global elements of $g$ ranges over all possible global elements of $f^*$ both the $a_i$ and the $b_j$ have corresponding global elements of $g^*$ in $\mathcal{G}$. We are going to prove the two decomposition of $F$ can be refined by a common one whose image by $h^*$ must be equal to both images. For each pair $(i,j) \in I \times J$, one can take the global element $(a_i, b_j) : 1_\mathcal{F} \rightarrow f^*(E_i \times E_j)$ induced by the universal property of the product for is preserved by $f^*$; by composition and cancellation of pullbacks, we know that the morphism $ u_i \times v_j : D_i \times H_j \rightarrow E_i \times E_j$ is the pullback over $ E_i \times E_j$ of the pullbacks of $ u_i$ and $v_j$ along the corresponding projections; then take the following fiber
\[\begin{tikzcd}
	{(a_i,b_j)^*f^*(D_i \times H_j)} & {f^*(D_i \times H_j)} \\
	{1_\mathcal{F}} & {f^*(E_i \times E_j)}
	\arrow["{(a_i, b_j)}"', from=2-1, to=2-2]
	\arrow["{f^*(u_i \times v_j)}", from=1-2, to=2-2]
	\arrow["{(a_i,b_j)^*f^*(u_i \times v_j)}"', from=1-1, to=2-1]
	\arrow[from=1-1, to=1-2]
	\arrow["\lrcorner"{anchor=center, pos=0.125}, draw=none, from=1-1, to=2-2]
\end{tikzcd}\]
But stability of colimits ensures that this later exactly is the fiber of the inclusions respective to the two altenrative colimit decompositions at the pair $ (i,j)$
\[\begin{tikzcd}
	{(a_i,b_j)^*f^*(D_i \times H_j)} && {b_j^*f^*H_j} \\
	{a_i^*f^*D_i} & {\underset{i \in I}{\colim} \; a_i ^*f^*D_i} & {\underset{j \in J}{\colim} \; b_j ^*f^*H_j}
	\arrow[from=1-1, to=2-1]
	\arrow[from=1-1, to=1-3]
	\arrow["{q'j}", from=1-3, to=2-3]
	\arrow["{q_i}"', from=2-1, to=2-2]
	\arrow[Rightarrow, no head, from=2-2, to=2-3]
	\arrow["\lrcorner"{anchor=center, pos=0.125}, draw=none, from=1-1, to=2-2]
\end{tikzcd}\]
so that one has a new colimit decomposition of $F$ as a two steps colimit
\[ F \simeq \underset{(i,j) \in I \times J}{\colim}  \; (a_i,b_j)^*f^*(D_i \times H_j)  \]
But in fact, would one now fix either $i$ or $j$ and allow to range over the other variable, stability of colimits gives a decomposition

\[ a_i^*f^*D_i \simeq  \underset{j \in J}{\colim} \; (a_i,b_j)^*f^*(D_i \times H_j) \hskip1cm b_j^*f^*H_j \simeq  \underset{i \in I}{\colim} \; (a_i,b_j)^*f^*(D_i \times H_j)  \]
Now for each $(i,j) \in I \times J$ both $a_i$ and $b_j$ correspond to global elements $\overline{a_i} : 1_\mathcal{G} \rightarrow g^*E_i$ and $\overline{b_j} : 1_\mathcal{G} \rightarrow g^*E_j$, and so does $ (a_i,b_i)$ which defines an element $ \overline{(a_i,b_i)} : 1_\mathcal{G} \rightarrow g^*(E_i \times E_j)$ in a way that is compatible with the projections by naturality of the cone. Hence, one can take at each such pair $ (i, j)$ the following fiber
\[\begin{tikzcd}
	{\overline{(a_i,b_j)}^*g^*(D_i \times H_j)} & {g^*(D_i \times H_j)} \\
	{1_\mathcal{G}} & {g^*(E_i \times E_j)}
	\arrow[from=1-1, to=2-1]
	\arrow[from=1-1, to=1-2]
	\arrow["{\overline{(a_i,b_j)}}"', from=2-1, to=2-2]
	\arrow["{g^*(u_i \times v_j)}", from=1-2, to=2-2]
	\arrow["\lrcorner"{anchor=center, pos=0.125}, draw=none, from=1-1, to=2-2]
\end{tikzcd}\]
Then commutation of colimits ensures that the value of $h^*$ computed on both colimit decomposition must agree because of this common refinement
\[ \underset{i \in I}{\colim} \; \overline{a_i}^*g^*D_i \simeq \underset{(i,j) \in I \times J}{\colim} \; \overline{(a_i,b_j)}^*g^*(D_i \times H_j) \simeq \underset{j \in J}{\colim} \; \overline{b_j}^*g^*H_j
\]

We let the reader convince himself of the functoriality of this construction, which relies on the universal property of colimit. 
We also have a natural equivalence $ h^*f^* \simeq g^*$, for in the case where $ F$ is simply of the form $ f^*E$, the colimit decomposition condenses into a single object and one simply take $ h^*(f^*E) = g^*E$ according to our construction.

Also commutation of colimits ensures that the functor $ h^* : \mathcal{F} \rightarrow \mathcal{G}$ thus defined is cocontinuous. It is also lex: on one hand, it preserves the terminal object as $ h^*(1_\mathcal{F}) \simeq h^*f^*(1_\mathcal{E}) \simeq g^*1_\mathcal{E}$, so it suffices to show it preserves pullbacks, but this uses a similar argument of fibers decomposition and stability of colimits. Hence, we do have a geometric morphism $ h : \mathcal{G} \rightarrow \mathcal{F}$ which is morevoer part of a morphism $ g \rightarrow f$ in the pseudoslice $\GTop/\mathcal{E}$. 

Also, one recovers the cone $(\overline{a})_{(E,a) \in 1_\mathcal{F} \downarrow f^*}$ as $ \overline{a} \simeq h^*(a)$ by functoriality of $h^*$ applied to each element $a : 1_\mathcal{F} \rightarrow f^*E$, combining the previous calculation of $h^*f^*$ with the calculation at the terminal object. 

Finally, let us check the uniqueness part of the universal property of the bilimit. In one sense, two cones inducing the same $h$ must consist of the same global elements given by $h^*$. Dually suppose one has two $h,h'$ inducing the same cone, that is, agreeing on the global elements. As $\mathcal{F}$ is generated under global elements, the colimit decomposition as fibers ensures that $h^*$ and $h'^*$ actually coincide on every objects. To conclude, let us observe that there is no 2-dimensional part to check in the universal property of the bilimit as etale morphisms are discrete, so that there is no morphism of cone to test the universal property at. \end{proof}

\subsection{Opfibrational aspects}


Pro-etale geometric morphisms have an interesting intrinsic property inside of the 2-category of topoi, which manifests also at the level of their points: they are \emph{discrete opfibrations}. It is easy to show that etale morphisms are opfibrations, and as we shall see below, this property is inherited by bilimits of such morphisms; moreover we saw at \cref{proetale are discrete} that pro-etale geometric morphisms are discrete in the 2-dimensional sense.  This suggests a relation with the \emph{comprehensive factorization system} on $\Cat$ introduced by \cite{street2010comprehensive}, or more exactly, is dual version (\emph{initial functor}, \emph{discrete opfibration}).\\ 

We already met initial functors at \cref{initial functors}; on the other hand, the definition of discrete opfibrations in $\Cat$ is rather standard; let us emphasize that, as often with right classes in $\Cat$, discrete opfibrations admit a representable analog in any 2-category:

\begin{definition}
    Let $\mathcal{K}$ be a 2-category; a morphism $ f : C \rightarrow D$ is an \emph{opfibration} in $\mathcal{K}$ if for any $ B$ in $\mathcal{K}$, the corresponding functor $ \mathcal{K}[B, f] : \mathcal{K}[B, C] \rightarrow \mathcal{K}[B, D]$ is an opfibration in $\Cat$. Moreover, we shall say that such a representable opfibration is \emph{discrete} if it is a discrete object in $\mathcal{K}/D$.
\end{definition}

\begin{proposition}\label{proetale are dopfib}
    Pro-etale geometric morphisms are discrete opfibrations in $\GTop$. 
\end{proposition}

\begin{proof}
 First, observe that etale geometric morphisms are representable discrete opfibrations: for any 2-cell as below 
\[\begin{tikzcd}
	&& {\mathcal{E}/E} \\
	{\mathcal{F}} && {\mathcal{E}}
	\arrow[""{name=0, anchor=center, inner sep=0}, "f"{description}, from=2-1, to=2-3]
	\arrow["{\pi_E}", from=1-3, to=2-3]
	\arrow[""{name=1, anchor=center, inner sep=0}, "a", from=2-1, to=1-3]
	\arrow[""{name=2, anchor=center, inner sep=0}, "g"', curve={height=18pt}, from=2-1, to=2-3]
	\arrow["\phi", shorten <=2pt, shorten >=2pt, Rightarrow, from=0, to=2]
\end{tikzcd}\]
the opcartesian lift of $\phi $ is provided by the morphism of global element defined as the composite $ \phi^\flat_{E} a : 1_\mathcal{F} \rightarrow g^*E$. Similarly, for a cofiltered diagram $ \mathbb{E} : I \rightarrow \mathcal{E}$, if one replaces in the 2-cell above the etale morphism by a pro-etale morphism, then $a$ defines a cofiltered family of global elements $ a_i : 1_\mathcal{F} \rightarrow f^*E_i$ and it suffices to compose each $ a_i$ with the corresponding component $ \phi^\flat_{E_i}$ to have a cartesian lift -- functoriality being provided by pasting each morphism of element at $ d$ with the corresponding naturality square of $ \phi^\flat$ at $u_d$. Moreover, we saw that pro-etale morphisms over $\mathcal{E}$ are discrete in $\GTop/\mathcal{E}$. 
\end{proof}

This property is preserved by the functor of points $ \pt : \GTop \rightarrow \Set$. Recall that a point of the etale topos at $E$ is the same as a pair $(x,a)$ with $x$ a point of $\mathcal{E}$ together with $a$ an element of the fiber $ x^*E$. In fact, each object $E$ defines an evaluation functor $\ev_E : \pt(\mathcal{E}) \rightarrow \Set$ sending a point $x$ to its fiber at $E$, which exhibits $\pt(\pi_E)$ as the discrete opfibration 
\[\begin{tikzcd}
	{\pt(\mathcal{E}/E) \simeq \displaystyle\int \ev_E } & {\pt(\mathcal{E})}
	\arrow["{d_E}", from=1-1, to=1-2]
\end{tikzcd}\]

The functor $ \pt : \GTop \rightarrow \Cat$ preserves bilimits for it is representable: hence for a pro-etale geometric morphism for a diagram $ (\mathcal{E}/E_i)_{i \in I}$ we have
\[ \pt(\underset{i \in I}{\bilim} \; \mathcal{E}/E_i) \simeq \underset{i \in I}{\bilim} \; \pt(\mathcal{E}/E_i) \]

This exactly means that a point of the pro-etale topos is the same as a triples of families $((x_i, \alpha_d)_{i \in I, d \in I^2}, (a_i)_{i\in I}$ with $(x_i, \alpha_d)_{i \in I, d \in I^2}$ a families of points $ x_i : \Set \rightarrow \mathcal{E}$ equiped with descent isomorphisms $\alpha_d : x_i \simeq x_j$ for each $d:i \rightarrow j$ in $I$ satisfying the cocycles identities, together with elements of the corresponding fibers $a_i \in x_i^*E_i$ satisfying the descent property that for any $d : i \rightarrow j$, with $u_d : E_i \rightarrow E_j$ the corresponding morphism in $\mathcal{E}$, the image of $a_i$ through either member of the following naturality square is equal to $a_j$:
\[\begin{tikzcd}
	{x_i^*(E_i)} & {x_i^*(E_j)} \\
	{x_j^*(E_i)} & {x_j^*(E_j)}
	\arrow["{x_i^*(u_d)}", from=1-1, to=1-2]
	\arrow["{\alpha^d_{E_j} \atop \simeq}", from=1-2, to=2-2]
	\arrow["{\alpha^d_{E_i} \atop \simeq}"', from=1-1, to=2-1]
	\arrow["{x_j^*(u_d)}"', from=2-1, to=2-2]
\end{tikzcd}\]
Those observations can be summed into the following:

\begin{proposition}
The functor $ \pt$ sends pro-etale geometric morphisms to discrete opfibrations. 
\end{proposition}

\begin{remark}
Beware here that we use the relaxed notion of discrete opfibration where we require fibers to be actually \emph{equivalent} to discrete categories than discrete themselves. In fact, as the points $x_i$ in such a coherent family as above are all isomorphic, there is no distinguished choice of a point, only of the descent isomorphisms. Hence, the fibers of a point of $\mathcal{E}$ does not just contains a discrete family of elements of its fibers at the $E_i$, but actually form a groupoid whose connected components are coherent families of points isomorphic to it, together with elements of their fibers. This is because we allowed pro-etale geometric morphisms to be \emph{bilimits} of etale morphisms, which in $\Cat$ begets pseudolimits, rather than strict limits. 
\end{remark}

%% file: Sections/Section_5.tex
\section{Properties of terminally connected geometric morphisms}

Dually, we should give a few additional properties about terminally connected geometric morphisms. In this sections, we first give miscelaneous properties and characterizations, either from the point of view of the inverse image or its left pro-adjoint. After some site theoretic criteria, we turn to their stability properties, in particular to special laxness properties completing those established in \cref{laxness conditions}, and conlude one some conjectures about their relation with initial functors.

\subsection{General properties}

Terminally connected geometric morphisms are those which are connected from the point of view of the terminal object; hence, whenever the terminal object is generating, this is enough to ensure connectedness: as $ \Set$ precisely is the only topos having this property, this means that there is no distinction between connected topos and terminally connected topos:

\begin{proposition}\label{terco over set implies connected}
For any Grothendieck topos $\mathcal{E}$, the terminal geometric morphisms $ \gamma_\mathcal{E} : \mathcal{E} \rightarrow \Set$ is terminally connected if and only if it is connected.
\end{proposition}

\begin{proof}
    $\gamma_\mathcal{E}$ is terminally connected if one has a natural isomorphism $ \mathcal{E}[1_\mathcal{E}, \gamma_\mathcal{E}^*] \simeq \Set[1, -]$ but for $ 1$ is generating in $\Set$, the later is the identity functor of $ \Set$, while $\mathcal{E}[1_\mathcal{E}, \gamma_\mathcal{E}^*] = \Gamma_\mathcal{E} \gamma_\mathcal{E}^*$: hence the natural isomorphism witnessed by terminal connectedness provides a natural isomorphism $ \Gamma_\mathcal{E} \gamma_\mathcal{E}^* \simeq \id_\Set$ which is an inverse of the unit of the $ \gamma_\mathcal{E} \dashv \Gamma_\mathcal{E}$ adjunction: this forces $ \gamma^*_\mathcal{E}$ to be fully faithful, hence $ \mathcal{E}$ to be connected. 
\end{proof}

It is known that a locally connected morphism $f$ is connected if and only if the global left adjoint $ f_!$ preserves the terminal object. However, in the case of an essential geometric morphism, this condition only amounts to being a terminally connected morphism, and indeed, the latter were originally defined as such in \cite{caramello2020denseness}[4.7]. In absence of the global left adjoint, we can still nevertheless say the following:

\begin{proposition}
    Let be $f : \mathcal{F} \rightarrow \mathcal{E}$ a geometric morphism. Then $ f$ is terminally connected iff $f_! : \mathcal{F} \rightarrow \Pro(\mathcal{E})$ preserves the terminal object. 
\end{proposition}

\begin{proof}
Recall from \cref{free procompletion} that there is a fully faithful embedding $ \iota_\mathcal{E} :\mathcal{E} \hookrightarrow \Pro(\mathcal{E})$, which sends in particular $1_\mathcal{E}$ to the terminal pro-object as stated in \cref{terminal proobject}. For any object $E$ in $\mathcal{E}$ we hence have an equivalence
\[ \Pro(\mathcal{E}) [ 1_{\Pro(\mathcal{E})}, \iota_\mathcal{E}(E)] \simeq \mathcal{E}[1_\mathcal{E}, E] \]
Hence, terminal connectedness of $f$ whenever $ f_!(1_\mathcal{F}) = 1_{\Pro(\mathcal{E})}$ as then
\[  \mathcal{F}[1_\mathcal{F}, f^*E] \simeq \Pro(\mathcal{E}) [ 1_{\Pro(\mathcal{E})}, \iota_\mathcal{E}(E)] \]
Conversely, if $f$ is terminally connected, then $f_!1_\mathcal{F}$ and $1_{\Pro(\mathcal{E})}$ are identified left to any object of the form $\iota_{E}$. But $ \Pro(\mathcal{E})$ is cogenerated from $\mathcal{E}$ by cofiltered limits: hence we must have $ f_!1_\mathcal{F} \simeq 1_{\Pro(\mathcal{E})}$
\end{proof}

We can also give a dual of \cref{characterization of pro-etales}: as well as pro-etale geometric morphisms are those whose codomain is generated under colimits from fibers of global elements, terminally connected are those whose inverse image is closed under such colimits of fibers:

\begin{proposition}
    A geometric morphism $f : \mathcal{F} \rightarrow \mathcal{E}$ is terminally connected if and only if any any colimit of fibers is actually in the image of $f^*$, and any arrow induced by such colimits from a morphism of diagram with transition morphism in the image of $t_f^*$ is also in the image of $f^*$. 
\end{proposition}

\begin{proof}
In one direction, take a global element $ a : 1_\mathcal{F} \rightarrow f^*E$: t
Let $F $ decompose as a colimit $F \simeq \colim_{i \in I} a_i^*f^*D_i$ with $ h_i : D_i \rightarrow E_i$ in $\mathcal{E}$. Then if $f$ is terminally connected, the $a_i$ arise uniquely from global elements $\overline{a_i}$ in $\mathcal{E}$, and we could then compute already the colimit of the corresponding fibers in $\mathcal{E}$ before applying $f^*$ for it preserves both colimits and pullbacks:
\begin{align*}
     F &\simeq \underset{i \in I}{\colim}\; a_i^*f^*D_i \\ 
       &\simeq \underset{i \in I}{\colim}\; f^*\overline{a_i}^*D_i \\ 
       &\simeq f^*\underset{i \in I}{\colim}\; \overline{a_i}^*D_i
\end{align*}
Hence, $F$ lies in the essential image of $f^*$. 
\end{proof}

Some properties of connected morphisms relative to coproducts and cardinals partially generalize. First, recall that $f$ being connected amounts for all the units $\eta_E$ to being a natural isomorphisms; in terminal connectedness, this is only ``globally true" in the sense that $ \Gamma_\mathcal{E}$ naturally inverts the unit $ \eta$, which can be rephrased as:

\begin{proposition}\label{1 is orthogonal}
    A geometric morphism $ f : \mathcal{F} \rightarrow \mathcal{E}$ is terminally connected iff $1_\mathcal{E}$ is $\eta$-orthogonal. 
\end{proposition}

\begin{proof}
   Any $ a : 1_\mathcal{E} \rightarrow f_*f^*E$ corresponds to a unique $ \overline{a} : 1_\mathcal{F} \rightarrow f^*E$ with $ a = \eta_{E} f_*(a)$, but if one can lift uniquely $\overline{a}$ to some $b : 1_\mathcal{E} \rightarrow E$ then automatically from \cref{relation with units} we are done. Conversely, if one has orthogonality, for any $a : 1_\mathcal{F} \rightarrow f^*E$ the induced $ f_*(a)$ lifts uniquely along $\eta_E$ and provides the desired lift. 
\end{proof}

\begin{remark}\label{lift of families} 
Another easy observation is that, though a terminally connected does not necessarily lift generalized elements in general, it lifts at least those indexed by locally constant sheaves (those of that are $\Set$-indexed coproducts of the terminal objects):  for $ X = \coprod_{i \in I} 1_\mathcal{F}$: then one has $ [X, f^*] \simeq [\coprod_{i \in I} 1_\mathcal{E}, -]$. In other words, one can lift uniquely $I$-indexed families of global elements of a same object, and in a natural way. 
\end{remark}

In particular, terminally connected morphisms fit in a pattern of connectedness involving full faithfulness of extensions or restrictions of the inverse image part relative to different classes of internal locales, from the more general to the more discrete: 

\begin{proposition}
    For a geometric morphism $f : \mathcal{F} \rightarrow \mathcal{E}$:
    \begin{itemize}
        \item $f$ is hyperconnected iff $ f^* : \Frm(\mathcal{E}) \rightarrow \Frm(\mathcal{F})$ is fully faithful 
        \item $f$ is connected iff $ f^* : \disc\Frm(\mathcal{E}) \rightarrow \disc\Frm(\mathcal{F})$ is fully faithful
        \item if $f$ is terminally connected then $ f^* : \Card(\mathcal{E}) \rightarrow \Card(\mathcal{F})$ is fully faithful; 
        \item if $f$ is pure then $ f^* : \fin\Card(\mathcal{E}) \rightarrow \fin\Card(\mathcal{F})$ is fully faithful. 
    \end{itemize}
\end{proposition}

\begin{proof}
    The first item is \cite{elephant}[Proposition C2.4.15]; the second is immediate as discrete frames are exactly objects of the topos. For the third item, we saw at \cref{lift of families} that terminally connected geometric morphisms lift uniquely set-indexed families of global elements, which applies in particular to cardinals: if $ \alpha$ and $ \beta$ are two cardinals, with corresponding cardinals in $\mathcal{E}$ being given as the coproduct $ \coprod_\alpha 1_\mathcal{E}$ and $ \coprod_\beta 1_\mathcal{E}$, we have
    \begin{align*}
        \mathcal{F}[f^*\coprod_\alpha 1_\mathcal{E}, f^*\coprod_\beta 1_\mathcal{E}] &\simeq \mathcal{F}[\coprod_\alpha 1_\mathcal{F}, f^*\coprod_\beta 1_\mathcal{E}] \\
        &\simeq \mathcal{E}[\coprod_\alpha 1_\mathcal{E}, \coprod_\beta 1_\mathcal{E}]
    \end{align*}
    The same idea applies to the pure case, but for finite cardinals. 
\end{proof}

We also have a dual property, which states that the direct image $ f_*$ of a terminally connected morphism ``globally" preserves cardinals:

\begin{proposition}
    If $f$ is terminally connected then for any $\alpha$ we have $ \Gamma_\mathcal{E}f_*\coprod_{\alpha} 1_\mathcal{F} \simeq \Gamma_\mathcal{E} \coprod_{\alpha} 1_\mathcal{F}$.
\end{proposition}

\begin{proof}
This simply relies on the natural isomorphism $ \Gamma_\mathcal{F}f^* \simeq \Gamma_\mathcal{E}$.
\end{proof}

\subsection{Site-theoretic aspects}

Let us now give a few site-theoretic criteria to test terminal connectedness. We first discuss corresponding properties for morphisms of sites, and then for comorphisms of sites. 

\begin{remark}
We should first emphasize that things are not as well behaved as one might think. Given a morphism of site $ f : (\mathcal{C},J) \rightarrow (\mathcal{D}, K) $ with terminal object, one could restrict the lifting condition of global elements to those ones arising from those sites: that is, that any global element $1_\mathcal{D} \rightarrow  f^*C$ should come from a unique global element $ 1_\mathcal{C} \rightarrow C$ in $\mathcal{C}$ in a natural way. This rises a double warning:\begin{itemize}
    \item asking $ f$ to lift uniquely global elements may not be sufficient in general to ensure that the associated geometric morphism $\Sh(f) : \Sh(\mathcal{D}, K) \rightarrow \Sh(\mathcal{C},J)$ is terminally connected. Indeed, although the sites generate the associated topoi through colimit, it may happen that the terminal object is too ``coarse" in the topos to retrieve global elements of arbitrary objects constructed as colimits of representables. Indeed, take a global element $ a : 1_\mathcal{\Sh(\mathcal{D},K)} \rightarrow \Sh( f)^*(F) $ with  $F = \colim_{i \in I} \mathfrak{a}_J \hirayo_{C_i}$ its colimit decomposition: even though $ \Sh(f)^*(\colim_{i \in I} \mathfrak{a}_J \hirayo_{C_i}) \simeq \colim_{i \in I} \mathfrak{a}_K \hirayo_{f(C_i)}$, using lifting of elements by $ f$ would first require to lift $a$ through some member of this colimit, which is not possible in all generality. We saw at \cref{cofinal cat of representable elements} that it is possible to find, for a given geometric morphism, suitable small sites on which testing terminal connectedness was possible, but this involved special sites where representable generated objects under filtered colimits for some cardinal against which the terminal object was compact. 
    \item conversely, terminal connectedness does not imply always the possibility to lift global elements inside of the given site: if one has $a : 1_{\mathcal{D}} \rightarrow f(C) $ in $\mathcal{D}$, then one can lift uniquely the corresponding element $ \mathfrak{a}_J \hirayo_a : 1_{\Sh(\mathcal{D},K)} \rightarrow \Sh(f)^*\mathfrak{a}_J \hirayo_C$ to some $ b : 1_{\Sh(\mathcal{D},K)} \rightarrow \mathfrak{a}_J \hirayo_C$ in $\Sh(\mathcal{C},J)$, but in absence of subcanonicity, such a global element may not arise from the site. However, this is exactly the only hindrance to do so: 
\end{itemize} 
\end{remark}

\begin{proposition}\label{terco implies lift in subcanonicity}
Let $ f : (\mathcal{C},J) \rightarrow (\mathcal{D}, K) $ be a morphism of site between site with terminal object such that $ \Sh(f)$ is terminally connected. If both $ (\mathcal{C},J)$ and $(\mathcal{D},K)$ are subcanonical, then $f$ lifts uniquely global elements from $\mathcal{D}$ to $\mathcal{C}$. 
\end{proposition}

\begin{proof}
    In the case where both sites are subcanonical, both $ \mathfrak{a}_J\hirayo$ and $ \mathfrak{a}_K\hirayo$ are fully faithful, and assuming $\Sh(f)$ to be terminally connected yields the following isomorphism in each $C$ of $\mathcal{C}$:
\begin{align*}
    \mathcal{D}[1_\mathcal{D}, f(C)] &\simeq \Sh(\mathcal{D},K)[ 1_{ \Sh(\mathcal{D},K)}, \Sh(f)^*(\hirayo_C)] \\
    &\simeq \Sh(\mathcal{C},J)[ 1_{ \Sh(\mathcal{C},J)},\hirayo_C]\\
    &\simeq \mathcal{C}[1_\mathcal{C}, C]
\end{align*}
\end{proof}

For the converse problem of deducing terminal connectedness from lifting of global elements from a site, we know at least two cases where it holds. The first one is cardinal-sensitive, and is actually close in mind of the argument we used at \cref{cofinal cat of representable elements}:

\begin{lemma}\label{test of lift on site}
Let $ f : \mathcal{F} \rightarrow \mathcal{E}$ be a geometric morphism; if one has a cardinal $\mu$ and a small standard site $(\mathcal{C},J)$ for $\mathcal{E}$ such that $ \mathcal{C}$ generates $\mathcal{E}$ under $ \mu$-filtered colimits and both $ 1_\mathcal{F}$ and $ 1_\mathcal{E}$ are $\mu$-compact, then $ f^*$ is terminally connected if and only if one has a natural isomorphism $ \mathcal{F}[1_\mathcal{F}, f^* \hirayo] \simeq \mathcal{C}[1_\mathcal{C}, -]$.   
\end{lemma}

\begin{proof}
If one has a site $(\mathcal{C},J)$ with all the conditions above, then one has a sequence of natural isomorphisms 
\begin{align*}
    \mathcal{F}[1_\mathcal{F}, f^*F] &\simeq \underset{\mathcal{C}\downarrow F}{\colim} \, \mathcal{F}[1_\mathcal{F},f^*C] \\
    &\simeq \underset{\mathcal{C}\downarrow F}{\colim} \, \mathcal{E}[1_\mathcal{E},C] \\
    &\simeq \mathcal{E}[1_\mathcal{E},F]
\end{align*}
\end{proof}

\begin{remark}\label{Sufficient site}
Beware that here the condition that $ 1_\mathcal{E}$ is also $ \mu$-compact is crucial. However, we can always choose the site $(C,J)$ in such a way that this additional condition is satisfied. Take directly $ \mu$ sufficiently large to ensure that $ \mathcal{E}$ is locally $\mu$-presentable and both $ 1_\mathcal{E}, 1_\mathcal{F}$ are $ \mu$-compact, and then take the closure of $ \mathcal{E}_\mu$ under finite limits in $\mathcal{E}$. 
\end{remark}

Even better is the case where the terminal object is infinitesimal - for accessibility conditions are no longer required. Recall that a site $(\mathcal{D},K)$ is said to be \emph{local} if the terminal object is $K$-indecomposable -- that is, the only covering sieve for $1_\mathcal{D}$ is the maximal sieve, see \cite{elephant}[C3.6.3 (d)]. Then for such a site, $\Sh(\mathcal{D},K)$ is a local topos. 

\begin{lemma}\label{testing terminal connectedness with local sites}
Let $ f: (\mathcal{C}, J) \rightarrow (\mathcal{D},K)$ be a morphism of sites between small lex sites with $ (\mathcal{D}, K)$ a local site. If $f$ lifts global elements then $\Sh(f) : \Sh(\mathcal{D},J) \rightarrow \Sh(\mathcal{C},J)$ is terminally connected. Moreover, if $ (\mathcal{C},J)$ and $ (\mathcal{D},K)$ both are subcanonical, the converse also holds.  
\end{lemma}

\begin{proof}
Suppose that $f$ lifts global elements, so that we have a natural isomorphism $ \mathcal{D}[1_\mathcal{D}, f] \simeq \mathcal{C}[1_\mathcal{C}, -]$. We prove this suffices to lift also global elements of $\Sh(f)^*$ that arise in $\Sh(\mathcal{D},K)$. Let be $a : 1_{\Sh(\mathcal{D},K)} \rightarrow \Sh(f)^*E$ with $E$ a object of $\Sh(\mathcal{C},J)$. Since $ \mathcal{C}$ generates $\mathcal{E}$ under colimits and $ 1_{\Sh(\mathcal{D},K)}$ is tiny by localness of ${\Sh(\mathcal{D},K)}$, there exists a factorization of the form
\[\begin{tikzcd}
	{1_{\Sh(\mathcal{D},K)}} & {\Sh(f)^*E} \\
	& {\Sh(f)^*\mathfrak{a}_J\hirayo_C}
	\arrow["a", from=1-1, to=1-2]
	\arrow["b"', from=1-1, to=2-2]
	\arrow["{\Sh(f)^*(x)}"', from=2-2, to=1-2]
\end{tikzcd}\]
which is moreover essentially unique in the sense that for any two such lifts are identified with $a$ in the colimit $ \colim_{(C,a) \in \mathfrak{a}_J\hirayo\downarrow E} \Sh(\mathcal{D},K)[1_{\Sh(\mathcal{D},K)}, \Sh(f)^*\hirayo_C] $.

As $ \Sh(f)^*$ restricts as $f$ from $\mathcal{C}$ to $\mathcal{D}$, one has $ \Sh(f)^*\mathfrak{a}_J\hirayo_C \simeq \mathfrak{a}_K \hirayo_{f(C)}$ 
It suffices then to prove one can lift uniquely global elements of the form $b: 1_{\Sh(\mathcal{D},K)} \rightarrow \mathfrak{a}_K \hirayo_{f(C)} $. Beware that such maps do not necessarily come from $ \mathcal{D}$ for the topology $ K$ is not supposed to be subcanonical so $\mathfrak{a}_K\hirayo$ may not be fully faithful. However, since $K$ covers are sent to canonical cover in $\Sh(\mathcal{D},K)$, and the terminal object comes from the terminal object of the site $ 1_{\Sh(\mathcal{D},K)} \simeq \mathfrak{a}_K\hirayo_{1_\mathcal{D}}$, we know that any such global element $b$ is induced through the universal property of colimits by the image of a $K$-cover $ (u_i : D_i \rightarrow 1_{\mathcal{D}})_{i \in I}$ together with a cocone $ (v_i : D_i \rightarrow f(C))_{i \in I}$ such that for any $i$ in $I$ one has
\[\begin{tikzcd}
	{\mathfrak{a}_K\hirayo_{D_i}} \\
	{1_{\Sh(\mathcal{D},K)}} & {\mathfrak{a}_K\hirayo_{f(C)}}
	\arrow["b"', from=2-1, to=2-2]
	\arrow["{\mathfrak{a}_K\hirayo_{u_i}}"', from=1-1, to=2-1]
	\arrow["{\mathfrak{a}_K\hirayo_{v_i}}", from=1-1, to=2-2]
\end{tikzcd}\]
But by localness of the site (equivalently, by localness of $\Sh(\mathcal{D},K)$), the identity of $1_{\mathcal{D}} $ factorizes through some $ u_i$ as $ 1_{\mathcal{D}} = u_iw $ in $\mathcal{D}$. This factorization is essentially unique in the sense that for any other factorization through $ 1_{\mathcal{D}}= u_{i'}w'$ then $ w,w'$ factorize through the pullback $D_i \times_D D_{i'} $ and are hence identified in the corresponding colimit in $\Sh(\mathcal{D},K)$. Hence, we get an element $ v_i w : 1_{\mathcal{D}} \rightarrow f(C)$ in $\mathcal{D}$, which comes uniquely from an element $ \overline{c}:  1_{\mathcal{C}} \rightarrow C$ in $\mathcal{C}$. For $ w $ is a section of $ u_i$, we have that $f(c) = b$. 
\end{proof}

One may ask dually what kind of \emph{comorphisms of sites} induce terminally connected geometric morphisms: this question was answered for the case of essential terminally connected geometric in \cite{caramello2020denseness} in term of an notion of \emph{cofinality relative to a Grothendieck topology}. 

\begin{definition}
A functor $ f: A \rightarrow \mathcal{C}$ is \emph{$J$-cofinal} relative to a Grothendieck topology $J$ on $\mathcal{C}$ if it satisfies the following two conditions:\begin{itemize}
    \item for any $c$ in $\mathcal{C}$, there is a $J$-cover $(u_i : C_i \rightarrow c)_{i \in I}$ and for each $i \in I$ an object $a_i$ of $A$ and an arrow $ v_i : c_i \rightarrow f(a_i)$;
    \item for any span $u: c \rightarrow f(a)$, $u': c \rightarrow f(a')$ in $\mathcal{C}$ there is a $J$-cover $(u_i : C_i \rightarrow c)_{i \in I}$ such that for each $i \in I$ the composites $ uu_i$ and $ u'u_i$ belong to the same connected component of $c_i \downarrow f$.   
\end{itemize} 
\end{definition}

\begin{proposition}[{\cite{caramello2020denseness}[Proposition 4.64]}]
Let $ f: (\mathcal{C}, J) \rightarrow (\mathcal{D},K)$ be a continuous comorphism of small generated sites. Then the induced geometric morphism $C_f : \Sh(\mathcal{C}, J) \rightarrow \Sh(\mathcal{D},K)$ is terminally connected if and only if $f$ is $ K$-cofinal.
\end{proposition}

\begin{remark}
    This result characterizes comorphism presentations of \emph{essential} terminally connected geometric morphisms, as continuous comorphisms are those which induce an essential geometric morphism, see \cite{caramello2020denseness}[theorem 4.20].
\end{remark}

\subsection{Stability properties}

Now, we turn to stability properties of terminally connected geometric morphism. Although terminally connected geometric morphisms are not expected to be stable under pullback, there is another kind of stability they do satisfy, which is also commonly encountered amongst other classes of geometric morphisms. We say that a squares in $\GTop$ as below
\[\begin{tikzcd}
	{\mathcal{H}} & {\mathcal{F}} \\
	{\mathcal{G}} & {\mathcal{E}}
	\arrow["q", from=1-1, to=1-2]
	\arrow["t", from=1-2, to=2-2]
	\arrow["p"', from=1-1, to=2-1]
	\arrow["f"', from=2-1, to=2-2]
\end{tikzcd}\]
satisfies the \emph{(right) Beck-Chevalley condition} if the canonical transformation $ t^*f_* \rightarrow q_*p^*$ is a natural isomorphism. 

\begin{lemma}
    Let be a square as above satisfying the Beck-Chevalley condition; then $p$ is terminally connected whenever $t$ is. 
\end{lemma}
\begin{proof}
    Take a global element $ a : 1_\mathcal{H} \rightarrow p^*G $: for $ 1_\mathcal{H} \simeq q^*1_\mathcal{F}$, this global element corresponds uniquely by adjunction to a global element $ a' : 1_\mathcal{F} \rightarrow q_*p^*G$, but the later is canonically isomorphic to $t^*f_*G$: for $t$ is terminally connected, $a'$ lifts uniquely to a global element $ a'': 1_\mathcal{E} \rightarrow f_*G$, whence again throuhg adjunction a global element $ a''': f^*1_\mathcal{E} \simeq 1_\mathcal{F} \rightarrow G$.  
\end{proof}
From this observation, we can improve a bit the stability condition along pullbacks. There is an important class of geometric morphisms known as \emph{tidy}, those $f : \mathcal{F} \rightarrow \mathcal{E} $ whose direct image $f_* $ preserves filtered $\mathcal{E}$-indexed colimits. We choose not to say more about this definition -- an explanation of which can be found at \cite{elephant}[Definition C3.4.2] ; let us rather recall the following property, which is \cite{elephant}[Theorem 3.4.7]: take a square of geometric morphisms as above, such that $f$ is tidy ; if the square is a pullback square, then not only $q$ is also tidy, but then the square satisfies the Beck-Chevalley condition. This property, conjointly with the previous lemma, immediately yields the following:

\begin{proposition}
    Terminally connected geometric morphisms are stable under pullback along tidy geometric morphisms. 
\end{proposition}

This is not the only stability property of terminally connected geometric morphisms. Not only, as we saw at \cref{terco are stable along bicocomma}, terminally connected are stable left to bicocomma squares, a property that is related to the lax orthogonality of the comprehensive factorization of geometric morphism, but in fact, they also enjoy a dual property relative to bicomma squares:

\begin{proposition}\label{terco are stable left to bicomma}
Terminally connected geometric morphisms are stable left to bicomma squares: that is, in a bicomma square as below in $\GTop$
\[\begin{tikzcd}
	{t \downarrow f} & {\mathcal{F}} \\
	{\mathcal{G}} & {\mathcal{E}}
	\arrow[""{name=0, pos=0.5, inner sep=0}, "t", from=1-2, to=2-2]
	\arrow["f"', from=2-1, to=2-2]
	\arrow[""{name=1, pos=0.45, inner sep=0}, "{p_1}"', from=1-1, to=2-1]
	\arrow["{p_2}", from=1-1, to=1-2]
	\arrow["{\phi_{t,f}}"', shorten <=6pt, shorten >=6pt, Rightarrow, from=0, to=1]
\end{tikzcd}\]
then $ p_1$ is terminally connected whenever $t $ is.
\end{proposition}

\begin{proof}
The bicomma can be presented as the sheaf topos over the small generated site obtained as the bicocomma in $\Lex$ of the canonical site (where $ J_{t,f}$ is the topology induced by bicocomma inclusions):
\[\begin{tikzcd}
	{(\mathcal{E}, J_\can)} & {(\mathcal{F}, J_\can)} \\
	{(\mathcal{G}, J_\can)} & {(t^* \uparrow f^*, J_{t,f})}
	\arrow[""{name=0, anchor=center, inner sep=0}, "f^*"', from=1-1, to=2-1]
	\arrow["t^*", from=1-1, to=1-2]
	\arrow["q_1"', from=2-1, to=2-2]
	\arrow["q_2", ""{name=1, anchor=center, inner sep=0}, from=1-2, to=2-2]
	\arrow["{\psi_{t^*,f^*}}"', shorten <=10pt, shorten >=10pt, Rightarrow, from=1, to=0]
\end{tikzcd}\]
By a pseudomonadicity argument, this bicolimit in $\Lex$ is actually the free cartesian category over the cocomma $t^* \uparrow f^*$ computed in $\Cat$; but the latter is the category defined as follows:\begin{itemize}
    \item its objects are the disjoint union of the objects of $\mathcal{F}$ and the objects of $\mathcal{G}$
    \item it has as arrows the arrows of $\mathcal{F}$ between objects of $\mathcal{F}$, the arrows of $\mathcal{G}$ between the objects of $\mathcal{G}$, and for each object of $\mathcal{E}$, one has a formal arrow $ \psi_E : t^*(E) \rightarrow f^*E$;
    \item then one adds all composites of the form 
\[\begin{tikzcd}
	F & {t^*E} & {f^*E} & G
	\arrow["a", from=1-1, to=1-2]
	\arrow["{\psi_E}", from=1-2, to=1-3]
	\arrow["b", from=1-3, to=1-4]
\end{tikzcd}\]
\item then one quotients by enforcing the composites $ ( \psi_E\id_{t^*E} $ and $ \id_{f^*E}\psi_E$ to be both equal to $ \psi_E$, and identifying any two composites $ (a, E, b)$ and $(a', E', b')$ that are related by a zigzag. 
\end{itemize}
At this point, the component of $\psi$ at the terminal object $1_\mathcal{E}$ cannot be enforced to be an isomorphism. However, it becomes terminal in the free lex completion (which is the bicocomma in the 2-category of lex categories), as the bicocomma inclusions have to preserve limits and in particular the terminal object on which they must agree: hence $ \psi_{1_\mathcal{E}}$ can be taken as the identity of the terminal object in $t^* \uparrow f^*$. 

Then we prove that the cocomma inclusion $q_1$ lifts global elements. This inclusion simply sends $ G$ to the object $ G$. From the previous consideration on the terminal object in the bicocomma, we know that global elements of $ q_1$ in the bicocomma may arise in two ways: either they are of the form $b : 1_\mathcal{G} \rightarrow G$, but as those come from $\mathcal{G}$ they are trivially lifted; either they arise as composite of the form
\[\begin{tikzcd}
	{1_\mathcal{F}} & {t^*E} & {f^*E} & G
	\arrow["a", from=1-1, to=1-2]
	\arrow["{\psi_E}", from=1-2, to=1-3]
	\arrow["b", from=1-3, to=1-4]
\end{tikzcd}\]
But in this case, for $t$ is terminally connected, there exists a unique global element $\overline{a} : 1_\mathcal{E} \rightarrow E$ with $a=t^*a$ in $\mathcal{E}$, but now the naturality of $\psi$ at $\overline{a}$ 
\[\begin{tikzcd}
	{1_\mathcal{F}} & {1_\mathcal{G}} \\
	{t^*E} & {f^*E}
	\arrow["a"', from=1-1, to=2-1]
	\arrow["{\psi_E}"', from=2-1, to=2-2]
	\arrow["{f^*\overline{a}}", from=1-2, to=2-2]
	\arrow["{\psi_{1_\mathcal{E}}}", Rightarrow, no head, from=1-1, to=1-2]
\end{tikzcd}\]
ensures that the composite $ bf^*a : 1_\mathcal{G} \rightarrow G$ is a convenient lift for the global element above.
Hence, we proved $ q_1$ to lift uniquely global elements; but this is the inverse image part of the bicomma projection $ p_1$, which is hence terminally connected. 
\end{proof}

\begin{remark}
This latter result leads us to the concluding remark of this work and a conjecture about the relation between our factorization system and the comprehensive factorization. In general, 2-dimensional factorization systems with comma-stable left classes are not that common: the only other instance we know in $\Cat$ is precisely the comprehensive factorization system (initial, discrete opfibration) ! Bicomma stability of terminally connected geometric morphisms remarkably mirrors the fact, first observed in lemma 0.6 of \cite{Shulmanlab}, that initial functors are themselves left stable under comma squares: that is, for a comma square as below in $\Cat$
\[\begin{tikzcd}
	{i \downarrow f} & {C} \\
	{B} & {D}
	\arrow[""{name=0, pos=0.5, inner sep=0}, "i", from=1-2, to=2-2]
	\arrow["f"', from=2-1, to=2-2]
	\arrow[""{name=1, pos=0.45, inner sep=0}, "{p_1}"', from=1-1, to=2-1]
	\arrow["{p_2}", from=1-1, to=1-2]
	\arrow["{\phi_{i,f}}"', shorten <=6pt, shorten >=6pt, Rightarrow, from=0, to=1]
\end{tikzcd}\]
then $ p_1$ is initial whenever $i $ is. It is not difficult to convince oneself of this fact: recall indeed that an object of the comma is a triple $ (c,b, u : i(c) \rightarrow f(b))$; if $i$ is initial, then for any $b$ in $B$, the category $ i \downarrow f(b)$ is non empty and connected. Hence, there is $ u : i(c) \rightarrow f(b)$ in $D$ from which $ p_1(c,b,u) = b$ ensures that $ p_1 \downarrow b$ is non empty (and even surjective on objects). Similarly, for $ i$ is initial, any two such $ (c, b, u) $ and $ (c',b, u')$ on the same $b$ are related by a zigzag in $C$. Hence, $p_1$ is initial.  
\end{remark}

\begin{remark}
At first sight, this suggests that terminally connected morphisms are initial in some internal sense. However it is not clear which version of initialness would be suited -- the reader may convince himself that the different possible notion of representable initialness are not very relevant. Left classes in factorization systems are indeed often more difficult to characterize.

Moreover, whatever initialness property they would fulfill, it would not be properly conveyed to ordinary initialness by the functor of points. We saw that pro-etale morphisms are sent to discrete opfibrations. However, it is not true that terminally connected morphisms induce initial functors between categories of points: take for instance the case of an essential geometric morphism $f$ and its terminally connected, etale factorization through $ \mathcal{E}/f_!1_\mathcal{F}$ corresponding to the unit $ \eta^!_{1_\mathcal{F}} : 1_\mathcal{F} \rightarrow f^*f_! 1_\mathcal{F}$ seen as a global element. A point of $\mathcal{E}/f_!1_\mathcal{F} $ is a pair $ (x,a)$ with $x$ a point of $\mathcal{E}$ and $a \in x^*(f_!1_\mathcal{F})$ an element of the fiber of $x^*$ at $f_!1_\mathcal{F}$. The functor $ \pt(t_f)$, where $t_f$ is the terminally connected part of the factorization, sends a point $y$ of $ \mathcal{F}$ to the pair $(\pt(f)(y), y^*(\eta^!_{1_\mathcal{F}}))$, where $\eta^!_{1_\mathcal{F}} : y^*1_\mathcal{F} \simeq 1_\mathcal{S} \rightarrow y^*f^*f_!1_\mathcal{F}$ is the image of the unit along $y^*$. As $\pt(\pi_{f_!1_\mathcal{F}})$ is a discrete fibration, all morphisms in $\pt(\mathcal{E}/f_!1_\mathcal{F}) $ are cartesian lifts of underlying morphisms between points in $\pt(\mathcal{E})$. Hence, if a point $ x$ of $\mathcal{E}$ has empty comma $\pt(f) \downarrow x$, then so is $ \pt(t_f) \downarrow x$, and the induced functor $ \pt(t_f)$ needs not be initial, though $t_f$ is terminally connected. 
\end{remark}

\begin{remark}
As disappointing as the previous remark may sound, we nevertheless conjecture a subtler relation between those two classes. We saw at \cref{terco are stable left to bicomma} that terminally connected geometric morphisms are stable left to bicomma squares in $\GTop$. Now if $f$ is terminally connected, for terminally connected are stable left to bicomma squares by \cref{terco are stable left to bicomma}, the bicomma topos $f \downarrow x$ at any point $x : \Set \rightarrow \mathcal{E}$ is terminally connected over $\Set$, hence a connected topos by \cref{terco over set implies connected}. For this is true at any point, this means in some sense that terminally connected geometric morphisms are ``topologically initial" in the sense that their initialness, though not witnessed by points nor in an internal way, is still manifest as a topological property of the bicomma object. We hope such a statement will be made more precise in a future work relying on a convenient notion of ``topological internalization".


\end{remark}